\documentclass[a4paper,11pt]{amsart}

\usepackage[dvipdfmx]{graphicx}
\usepackage[dvipdfmx]{color}
\usepackage{amsmath,amssymb}
\usepackage{amsthm}
\usepackage{mathrsfs}
\usepackage{enumitem}
\usepackage{mathtools}
\usepackage[all,cmtip,2cell]{xy}
\usepackage{array}
\usepackage{here}

\theoremstyle{plain}
\newtheorem{thm}{Theorem}[section]
\newtheorem{prop}[thm]{Proposition}
\newtheorem{lem}[thm]{Lemma}

\newtheorem*{thm*}{Theorem}
\newtheorem{prob}{Problem}[section]

\theoremstyle{definition}
\newtheorem{defi}[thm]{Definition}

\newtheorem*{NaC}{Notation and Conventions}
\newtheorem{nota}{Notation}
\newtheorem*{ACK}{Acknowledgement}

\theoremstyle{remark}
\newtheorem{rem}[thm]{Remark}

\newtheorem{ex}[thm]{Example}

\newtheorem{claim}{Claim}
\newtheorem{claimproof}{Proof of Claim}
\newtheorem*{claim*}{Claim}
\newtheorem*{claimproof*}{Proof of Claim}
\numberwithin{equation}{subsection}

\newcommand{\N}{\mathbb{N}}
\newcommand{\Z}{\mathbb{Z}}
\newcommand{\Q}{\mathbb{Q}}

\newcommand{\C}{\mathbb{C}}
\newcommand{\A}{\mathbb{A}}
\renewcommand{\P}{\mathbb{P}}
\newcommand{\F}{\mathbb{F}}

\newcommand{\ol}{\overline}

\newcommand{\wt}{\widetilde}

\let\div=\relax

\DeclareMathOperator{\Spec}{Spec}

\DeclareMathOperator{\Pic}{\mathrm{Pic}}
\DeclareMathOperator{\div}{\mathrm{div}}

\DeclareMathOperator{\eu}{\chi_{\mathrm{top}}}

\DeclareMathOperator{\Bl}{\mathrm{Bl}}

\DeclareMathOperator{\Sing}{\mathrm{Sing}}

\title[Fano compactifications]{Fano compactifications of contractible affine 3-folds with trivial log canonical divisors}
\author[M. Nagaoka]{Masaru Nagaoka}
\address{Graduate School of Mathematical Sciences\\The University of Tokyo\\3-8-1 Komaba\\Meguro-ku, Tokyo 153-8914, Japan}
\email{nagaoka@ms.u-tokyo.ac.jp}
\begin{document}
\maketitle
\begin{abstract}
T.\,Kishimoto raised the problem to classify all compactifications of contractible affine $3$-folds into smooth Fano $3$-folds with second Betti number two and classified such compactifications whose log canonical divisors are not nef.
In this article, we show that there are $14$ deformation equivalence classes of smooth Fano $3$-folds which can admit structures of such compactifications whose log canonical divisors are trivial. 
We also construct an example of such compactifications with trivial log canonical divisors for each of all the 14 classes.

\end{abstract}
\tableofcontents
\setcounter{section}{0}

\section{Introduction.}\label{sec:intro}
Throughout the present article we work over the field of complex numbers $\C$. In \cite{hi54}, F.\,Hirzebruch raised the problem to classify all compactifications of the affine space $\A^n$ into compact complex manifolds with second Betti number $B_2=1$ for all $n$. By contributions of M.\,Furushima, N.\,Nakayama, Th.\,Peternell and M.\,Schneider (cf. \cite{fur86, fur90, fur92, fur93, f-n89a, f-n89b, pe89, p-s88}),
the classification of compactifications of $\A^3$ into smooth projective 3-folds with $B_2=1$ was completed. 

Let $V$ be a smooth projective 3-fold with $B_2=1$ which contains $\A^3$ as an open set. Let $D\coloneqq V \setminus \A^3$. Then $V$ is a Fano 3-fold and the anticanonical divisor $-K_V$ is linearly equivalent to $rD$, where $r$ is the Fano index of $V$. First, in \cite{fur86, f-n89a, f-n89b, p-s88}, the pairs $(V, D)$ with $r \geq 2$ were classified. Then Furushima \cite{fur93} classified the pairs $(V, D)$ with $r=1$.

T.\,Kishimoto \cite{ki05} observed that their arguments make use of only the contractibility of $\A^3$ and that, by the arguments, it is possible to classify compactifications of contractible affine 3-folds into smooth Fano 3-folds with $B_2=1$. After that, Kishimoto raised the following problem as the corresponding problem in the case when $B_2=2$.

\begin{prob}\label{pr:main}
Classify triplets $(V, D_1 \cup D_2, U)$, where $V$ is a smooth Fano 3-fold with $B_2=2$, $D_1$ and $D_2$ are prime divisors on $V$, and $U$ is the complement $V \setminus (D_1 \cup D_2)$ which is a contractible affine 3-fold.
\end{prob}

We often call $D_1$ and $D_2$ \emph{the boundary divisors} and $K_V+D_1+D_2$ \emph{the log canonical divisor}.

Using the list \cite[Table 2]{m-m81}, Kishimoto classified all such triplets in \cite{ki05} when the log canonical divisors are not nef. For this reason, we consider such triplets whose log canonical divisors are nef. Then, by virtue of the Kawamata-Shokurov base point free theorem, the log canonical divisor for such a triplet is semiample and we denote by $\kappa(K_V+D_1+D_2)$ the dimension of the image of the morphism given by the linear system $\lvert m(K_V+D_1+D_2) \rvert$ for sufficiently large $m$.

In this article, we shall investigate triplets $(V, D_1 \cup D_2, U)$ as in Problem 1.1 satisfying the following condition:
\[
K_V + D_1 + D_2 \textup{ is nef and } \kappa(K_V + D_1 +D_2)=0.
\]
If the condition holds for a triplet $(V, D_1 \cup D_2, U)$, then $m(K_V+D_1+D_2)=0$ for sufficiently large $m$. Since the Picard groups of $V$ is torsion-free, we have $m=1$. Hence the condition is rewritten as
\[
(\dagger) \colon K_V+D_1+D_2=0.
\]
In this article, a triplet $(V, D_1 \cup D_2, U)$ is called \emph{a triplet with }$(\dagger)$ if it is as in Problem \ref{pr:main} and satisfies $(\dagger)$.
Note that a similar condition $K_V+D=0$ holds for every smooth projective compactification $(V, D)$ of $\A^3$ with $B_2=1$ and $r=1$.

Our main result is the following, where we denote by $\Q^3$ the smooth quadric 3-fold and by $V_5$ the smooth quintic del Pezzo 3-fold. 
\begin{thm}\label{thm:main}
\begin{enumerate} \item[]
\item[\textup{(1)}] Let $(V, D_1 \cup D_2, U)$ be a triplet as in Problem \ref{pr:main}. Suppose the triplet satisfies $(\dagger)$. Then there are only $14$ deformation equivalence classes to one of which the Fano 3-fold $V$ belongs. We give a precise description of $V$ below. 
\begin{enumerate}
\item[\textup{(A)}] If $V$ is imprimitive, then it is the blow-up of a Fano 3-fold $W$ along a smooth curve $C$ such that one of the following holds:
\begin{enumerate}
\item[\textup{(A1)}] $W \cong \P^3$ and $C$ is an elliptic curve of degree $3,4$.
\item[\textup{(A2)}] $W \cong \P^3$ and $C$ is a rational curve of degree $1,2,3,4$.
\item[\textup{(A3)}] $W \cong \Q^3$ and $C$ is a rational curve of degree $1,2,3,4$.
\item[\textup{(A4)}] $W \cong V_5$ and $C$ is a rational curve of degree $1,2,3$.
\end{enumerate}
\item[\textup{(B)}] If $V$ is primitive, then $V$ is one of the following:
\begin{enumerate} 
\item[\textup{(B1)}] $V \subset \P^2 \times \P^2$ is a divisor of bidegree $(1, 2)$.
\item[\textup{(B2)}] $V \cong \P^1 \times \P^2$.
\item[\textup{(B3)}] $V \cong \P_{\P^2}(\mathcal{O} \oplus \mathcal{O}(2))$, which is the blow-up of the cone over the Veronese surface at the vertex.
\end{enumerate}
\end{enumerate} 
\item[\textup{(2)}] For each $\mathcal{C}$ of $14$ classes as in \textup{(1)}, there is a triplet $(V, D_1 \cup D_2, U)$ as in Problem \ref{pr:main} satisfying $(\dagger)$ such that $V$ belongs to $\mathcal{C}$.
\end{enumerate}
\end{thm}

\begin{rem}
To prove Theorem \ref{thm:main}, we use the list \cite[Table 2]{m-m81} of smooth Fano 3-folds with $B_2=2$. Each type as in Theorem \ref{thm:main} corresponds to the numbers in [ibid.] as follows:
\begin{enumerate}
\item[(A1)] $V$ is of No.25 or No.28.
\item[(A2)] $V$ is of No.22, No.27, No.30 or No.33.
\item[(A3)] $V$ is of No.21, No.26, No.29 or No.31.
\item[(A4)] $V$ is of No.20, No.22 or No.26.
\item[(B1)] $V$ is of No.24.
\item[(B2)] $V$ is of No.34.
\item[(B3)] $V$ is of No.36.
\end{enumerate}
Note that No.22 and No.26 appear twice above.
\end{rem}
In this article, a triplet $(V, D_1 \cup D_2, U)$ with $(\dagger)$ is called \emph{of type $(*)$} if $V$ belongs to the type $(*)$ of Theorem \ref{thm:main}. 

Theorem \ref{thm:main} determines the candidate of the ambient spaces $V$ of triplets $(V, D_1 \cup D_2, U)$ with $(\dagger)$. In the forthcoming article \cite{n2}, we determine all triplets $(V, D_1 \cup D_2, U)$ with $(\dagger)$ of type (A1) or (A2) and it turns out that $U \cong \A^3$ except for one case.

This paper is structured as follows. 

In \S\ref{sec:pre}, we recall some facts about Gorenstein del Pezzo surfaces (\S\ref{sec:dP}) and topologies of varieties (\S\ref{sec:top}), which we use to prove Theorem \ref{thm:main}.

In \S\ref{sec:boundary}, we show that the linear equivalence classes of the boundary divisors of a triplet with $(\dagger)$ are uniquely determined up to permutation of the boundary divisors. In Lemma \ref{lem:euler}, we also prove an important equality between topological invariants about a triplet with $(\dagger)$, which is a key to prove Theorem \ref{thm:main} (1).

In \S\ref{sec:constA}, we construct an examples of triplets with $(\dagger)$ for each of all deformation equivalence classes of ambient spaces of type (A). To do so, we introduce a sufficient condition of triplets with $(\dagger)$ of type (A) in Lemma \ref{lem:const1}. We actually construct examples in \S\ref{sec:constA12}--\ref{sec:constA4}. We also show that the above condition is not a necessary condition in \S\ref{sec:counterexA}.

In \S\ref{sec:constB}, we construct an examples of triplets with $(\dagger)$ for each of remaining classes, i.e. of type (B). We review Kaliman's characterization of $\A^3$ (Theorem \ref{thm:2.26}) in \S\ref{sec:Kaliman} to construct examples explicitly in \S\ref{sec:constB3}--\S\ref{sec:constB12}. By combining \S\ref{sec:constA} and \S\ref{sec:constB}, we complete the proof of Theorem \ref{thm:main} (2).

In \S\ref{sec:excimpri}--\ref{sec:excpri}, we prove Theorem \ref{thm:main} (1) by showing that, unless the ambient space is of type (A) and (B), a triplet with $(\dagger)$ does not satisfy the equality of Lemma \ref{lem:euler}. \S\ref{sec:excimpri} (resp.\,\S\ref{sec:excpri}) deals with the case that the ambient space is imprimitive (resp.\,primitive).

\begin{NaC}
In this article, we always assume that Fano varieties are smooth.
We use the numbers assigned to Fano 3-folds with $B_2=2$ as in \cite[Table 2]{m-m81}. 
We also employ the following notation.
\begin{itemize}
\item $B_i(X)$: the $i$-th Betti number of a topological space $X$.
\item $\eu(X)$: the topological Euler number of a topological space $X$.
\item $\Sing\,X$: the singular locus of a variety $X$.
\item $K_X$: the canonical divisor of a Gorenstein variety $X$.
\item $p_a(r)$: the arithmetic genus of a curve $r$.
\item $E_f$: the exceptional divisor of a birational morphism $f$.
\item $f^{-1}_*(D)$: the strict transform of a divisor $D$ by a birational map $f$.
\item $\Q^3$: the smooth quadric 3-fold in $\P^4$.
\item $V_d$: a smooth del Pezzo 3-fold of degree $d$.
\item $\Q^2_0$: the quadric cone in $\P^3$.
\item $\F_n$: the Hirzebruch surface of degree $n$.
\item $f_n$: a ruling of $\F_n$.
\item $\Sigma_n$: the minimal section of $\F_n$.
\end{itemize}
\end{NaC}

\begin{ACK}
The author is greatly indebted to Prof.\,Hiromichi Takagi, his supervisor, for his encouragement, comments, and suggestions. He also wishes to express his gratitude to Prof.\,Takashi Kishimoto for his helpful comments and suggestions. He also would like to express his gratitude to Dr.\,Akihiro Kanemitsu and Dr.\,Takeru Fukuoka for giving very significant advice in order to make this article more readable.
\end{ACK}
\section{Preliminaries}\label{sec:pre}

\subsection{Gorenstein del Pezzo surfaces}\label{sec:dP}
In this subsection, we review some facts on Gorenstein del Pezzo surfaces.
\begin{thm}[cf.\cite{h-w81}, \cite{a-f83}]\label{thm:dP}
Let $S$ be a Gorenstein del Pezzo surface of degree $d \geq 3$.
\begin{enumerate}
\item[\textup{(1)}] If $S$ is normal and rational, then $\eu(S)=2 + B_2(S) \geq 3$.
\item[\textup{(2)}] If $S$ is normal and irrational, then $\eu(S)=1$.
\item[\textup{(3)}] If $S$ is non-normal, then $\eu(S)\geq 2$.
\item[\textup{(4)}] If $S$ is not a cone over a curve, then $\eu(S)\geq 3$.
\end{enumerate}
\end{thm}

\begin{lem}\label{lem:cone}
Let $S$ be a Gorenstein del Pezzo surface of degree $d \geq 4$ which is the cone over a curve. Then $S$ cannot be embedded in smooth 3-fold. 
\end{lem}

\begin{proof}
The surface $S \subset \P^d$ is an intersection of quadric hypersurfaces by \cite[Corollary 1.5]{fuj90}. Since $S$ is the cone over a curve, such quadric hypersurface is singular at the vertex of $S$. Then we can compute that the embedding dimension of $S$ at the vertex is $d \geq 4$. Hence we have the assertion. 
\end{proof}

\begin{nota}\label{no:dP} For a non-normal surface $S$, we use the following notation:
\begin{itemize}
\item $\sigma_S \colon \ol{S} \rightarrow S$: the normalization.
\item $\mathcal{C}_S \subset \mathcal{O}_S$: the conductor ideal of $\sigma_S$.
\item $E_S\coloneqq V_S(\mathcal{C}_S)$ and $\ol{E}_S\coloneqq V_{\ol{S}}(\sigma_S^*(\mathcal{C}_S))$: the subschemes defined by $\mathcal{C}_S$. We call $E_S$ (resp.\,$\ol{E}_S$) \emph{the conductor locus} of $S$ (resp.\,$\ol{S}$).
\end{itemize} 
\end{nota}

The following lemma is used in our proof of Propositions \ref{prop:No17} and \ref{prop:No14}.

\begin{lem}\label{lem:line}
Let $S$ be a non-normal Gorenstein del Pezzo surface of degree $d \geq 3$. Suppose that $S$ belongs to the class \textup{(C)} in \cite{a-f83} and $\ol{E}_{S}$ is reducible. Then each member $D \in \lvert -K_{S} \rvert$ satisfies $B_2(D) <d$.
\end{lem}

\begin{proof}
Since $S$ belongs to the class (C) of \cite{a-f83}, its normalization $\ol{S}$ is isomorphic to $\F_{d-2}$ and the conductor locus $\ol{E}_{S}$ of $\ol{S}$ is linearly equivalent to $\Sigma_{d-2}+f_{d-2}$. Since $\ol{E}_{S}$ is reducible, it consists of $\Sigma_{d-2}$ and a ruling $F$ of $\F_{d-2}$. The normalization $\sigma_S \colon \ol{S} \rightarrow S$ induces isomorphisms from both $\Sigma_{d-2}$ and $F$ to $E_S$.
We also have $-\sigma_S^*(K_S) \sim \Sigma_{d-2}+(d-1)f_{d-2}$.

As $B_2(D) \leq (D \cdot -K_S)=d$, it suffices to show that $B_2(D) \neq d$.
On the contrary, suppose that $B_2(D) = d$. 
Then $D$ consists of curves $r_1, \dots, r_d$ with $(r_i \cdot -K_S)=1$ for $i=1, \dots, d$. 

Fix $i \in \{1, \dots, d\}$. Then $r_i \not \subset E_S$ since $D$ is Cartier. 
As $\ol{E}_S=\Sigma_{d-2}+F$, we have ${\sigma^{-1}_{S*}}(r_{i}) \neq \Sigma_{d-2}$ or $F$.
If we write ${\sigma^{-1}_{S*}}(r_{i}) \sim a\Sigma_{d-2}+bf_{d-2}$, then it holds that 
$1=(r_i \cdot -K_S)_S=({\sigma^{-1}_{S*}}(r_{i}) \cdot -\sigma_S^{*}K_S)_{\ol{S}}=(a\Sigma_{d-2}+bf_{d-2} \cdot \Sigma_{d-2}+(d-1)f_{d-2})_{\ol{S}}=a+b.$ 
Hence ${\sigma^{-1}_{S*}}(r_{i}) \sim f_{d-2}$ for $i=1, \dots, d$. However this implies that $D$ is disconnected since $\sigma_S$ is as stated, which contradicts the ampleness of $D$.
\end{proof}

\subsection{Topologies of varieties}\label{sec:top}
In this subsection, we summarize some facts about the topologies of varieties which we apply in our proof of Propositions \ref{prop:5.2} and \ref{prop:5.3}.

\begin{lem}\label{lem:2.18}
Let $X$ be a smooth contractible affine 3-fold, $S \subset X$ a closed irreducible normal surface and $r \subset S$ a closed irreducible smooth curve. Let $\varphi\colon \mathrm{Bl}_{r}X \rightarrow X$ be the blow-up of $X$ along $r$. Suppose that $N\coloneqq \mathrm{Bl}_{r}X \setminus \varphi^{-1}_*(S)$ is contractible. Then we have
\[
H_i(S \setminus \Sing\,S, \Z)=
\left\{
\begin{array}{ll}
\Z &i=0\\
H_1(r \setminus (r \cap \Sing\,S), \Z)&i=1\\
\Z^{\sharp\Sing\,S}&i=3\\
0&\mathrm{otherwise.}
\end{array}
\right.
\]
\end{lem}

\begin{proof}
Let $M\coloneqq {X \setminus \Sing\,S}, S'\coloneqq S \setminus \Sing\,S$ and $E\coloneqq E_{\varphi} \cap N$. Note that $N=\left(\mathrm{Bl}_{r \setminus (r \cap \mathrm{Sing}\:S)}M \right) \setminus \varphi^{-1}_*(S')$.
Applying the Thom isomorphism to the pair $(X, \Sing\,S)$, we get the following exact sequence:
\[
\xymatrix@=10pt{
\cdots \ar [r]&{H_{i-5}(\Sing\,S, \Z)} \ar [r]&{H_{i}(M, \Z)} \ar [r]&H_{i}(X, \Z) \ar [r]& {H_{i-6}(\Sing\,S, \Z)} \ar [r] & {\cdots}. 
}
\]
Note that $N$ is contractible by the assumption and $E$ is an $\A^1$-bundle over $r \setminus (r \cap \Sing\,S)$. Thus we can calculate the singular homologies of $M$, $N$, and $E$. Then we have
\begin{eqnarray*}
\begin{split}
&
H_i(M, \Z)=
\left\{
\begin{array}{ll}
\Z &i=0\\
\Z^{\sharp\Sing\,S}&i=5\\
0&i \colon \mathrm{otherwise},
\end{array}
\right.
H_i(N, \Z)=
\left\{
\begin{array}{ll}
\Z &i=0\\
0&i \colon \mathrm{otherwise},
\end{array}
\right.
\\
&
H_i(E, \Z)=
\left\{
\begin{array}{ll}
\Z &i=0\\
H_1(r \setminus (r \cap \Sing\,S), \Z)&i=1\\
0&i \colon \mathrm{otherwise}.
\end{array}
\right. \hfill (*)
\end{split}
\end{eqnarray*}
As $\varphi\restriction_{N \setminus E} \colon N \setminus E \rightarrow M \setminus S'$ is an isomorphism, the Thom isomorphism to the pairs $(M, S')$ and $(N, E)$ gives us the following commutative diagram with the exact rows:
\[
\xymatrix@=10pt{
{\cdots} \ar [r]&{H_{i-1}(E, \Z)} \ar [r] \ar [d]&{H_{i}(N \setminus E, \Z)} \ar [r] \ar [d]^{\wr}&{H_{i}(N, \Z)} \ar [r] \ar [d]& {H_{i-2}(E, \Z)} \ar [r] \ar [d]&{\cdots}\\
{\cdots}\ar [r]&{H_{i-1}(S', \Z)} \ar [r]&{H_{i}(M \setminus S', \Z)} \ar [r]&{H_{i}(M, \Z)} \ar [r]& {H_{i-2}(S', \Z)} \ar [r] &{\cdots}. 
}
\]
Substituting the terms of the commutative diagram by $(*)$, we have the assertion.
\end{proof}

\begin{lem}\label{lem:2.17}
Let $f\colon \wt{X} \rightarrow X$ be a proper morphism of projective algebraic varieties such that its restriction $f^{-1}(U) \rightarrow U$ over a dense open subset $U$ is isomorphic. Let $Y \coloneqq X \setminus U$ and $\wt{Y} \coloneqq f^{-1}(Y)$, then there exists an exact sequence of cohomologies:
\[
\xymatrix@=10pt{
\cdots \ar [r]
&{H^{i}(X, \Z)} \ar [r]
&H^{i}(Y, \Z) \oplus H^{i}(\wt{X}, \Z) \ar [r] 
&{H^{i}(\wt{Y}, \Z)} \ar [r]
&{H^{i+1}(X, \Z)} \ar [r]
& {\cdots}. 
}
\]
\end{lem}

\begin{proof}
This follows by the same method as in the proof of \cite[Lemma 8.1.7]{ishi14}.
\end{proof}

\begin{defi}
For a normal surface singularity $(X,x)$, $\pi_{X, x} \coloneqq \displaystyle{\varprojlim_{U}}$ $ \pi_1(U \setminus \{ x \})$ is called \emph{the local fundamental group} of $(X, x)$. We say that $\pi_{X, x}$ is $\textit{perfect}$ if its abelianization is trivial.
\end{defi}

\begin{thm}[cf. {\cite[Satz 2.8]{br68}}]\label{thm:2.20} Let $(X, x)$ be a normal surface singularity. Then the following are equivalent:
\begin{enumerate}
\item[\textup{(1)}] $(X, x)$ is a quotient singularity.
\item[\textup{(2)}] The local fundamental group $\pi_{X, x}$ is finite.
\end{enumerate}
\end{thm}

\begin{thm} [cf. \cite{br68}, {\cite[Theorem1.4(a)]{b-d89}}]\label{thm:2.21}
Let $(X, x)$ be a rational double point. Then the following are equivalent:
\begin{enumerate}
\item[\textup{(1)}] $(X, x)$ is an $E_8$-singularity.
\item[\textup{(2)}] The local fundamental group $\pi_{X, x}$ is perfect.
\end{enumerate}
\end{thm}

\begin{thm} [cf. \cite{kptr89}]\label{thm:2.22}
Let $G$ be a reductive algebraic group acting algebraically on a contractible affine variety $X$. Then the algebraic quotient $X /\!/ G$ is also contractible. 
\end{thm}
\section{Boundary divisors}\label{sec:boundary}
In this section, we study properties of the boundary divisors of triplets with $(\dagger)$.
\begin{nota}\label{no:Fano} For a Fano 3-fold $V$ with $B_2=2$, we use the following notation for $i=1, 2$:
\begin{itemize} 
\item $\varphi_i$: the extremal contractions of $V$.
\item $W_i$: the image of $\varphi_i$.
\item $H_i$: the pullback of the ample generator of $\Pic W_i$.
\item $\mu_i$: the length of $\varphi_i$, i.e. $\mu_i\coloneqq \mathrm{min}\{(-K_V \cdot l) \in \N \:|\:l:$ a curve such that $\varphi_i(l) \textup{ is a point} \}$.
\item $l_i \subset V$: a curve such that $(-K_V \cdot l_i)=\mu_i$ and $\varphi_i(l) ${ is a point}. 
\end{itemize}
\[
\xymatrix@=10pt{
&V \ar[dl]_-{\varphi_1} \ar[dr]^-{\varphi_2}
&\\
W_1
&
&W_2
}
\]
\end{nota}

\begin{thm}[cf. {\cite[Theorem 5.1]{m-m83}}]\label{thm:2.3} With Notation \ref{no:Fano}, we have the following:
\begin{enumerate}
\item[\textup{(1)}] $\Pic V=\Z[H_{1}] \oplus \Z[H_{2}]$.
\item[\textup{(2)}] $-K_V \sim \mu_{1} H_{2} + \mu_{2} H_1$.
\item[\textup{(3)}] $(H_i \cdot l_j)=1- \delta_{ij}$ for all $i,j$, where $\delta_{ij}$ is Kronecker's delta.
\end{enumerate}

\end{thm}

In the remaining of this section, we fix a triplet $(V, D_1 \cup D_2, U)$ with $(\dagger)$.

\begin{lem}[cf. {\cite[Lemma.2.1]{ki05}}]\label{lem:2.1}
It holds that $\Pic V=\Z[D_1] \oplus \Z[D_2]$.
\end{lem}

\begin{lem}\label{lem:2.5}
Take $m_{ij} \in \Z$ such that $D_i \sim m_{i1} H_1 + m_{i2} H_2$. Then $m_{ij} \geq 0$ for all $i,j$.
\end{lem}

\begin{proof}
 By symmetry, we have only to show $m_{11} \geq 0$. Suppose that $m_{11}<0$. Then $m_{12} > 0$ since $D_1$ is a non-zero effective divisor. By the condition $(\dagger)$ and Theorem \ref{thm:2.3} (3), we have $m_{11}+m_{21} \geq 1$ and $m_{12}+m_{22} \geq 1$. Hence we have
\[
\begin{aligned}
m_{11}m_{22}-m_{12}m_{21} &\leq m_{11}m_{22}-m_{12}(-m_{11}+1) =m_{11}(m_{22}+m_{12})-m_{12}\\
 &\leq m_{11}-m_{12} \\
 &\leq -2,
\end{aligned}
\]
which implies that the matrix $(m_{ij})_{1 \leq i,j \leq 2}$ is not unimodular. This contradicts to Theorem \ref{thm:2.3} (1) and Lemma \ref{lem:2.1}. 
\end{proof}

\begin{prop}\label{prop:2.6} We have the following:
\begin{enumerate}
\item[\textup{(1)}] For every $k \in \{1,2\}$ such that $\mu_k=1$, we have $H_k \sim D_l$ for some $l$.
\item[\textup{(2)}] If $\mu_1, \mu_2 \geq 2$, then $V \cong \P^1 \times \P^2$.
\end{enumerate}
\end{prop}

\begin{proof}
Let us define $m_{ij}$ as in Lemma \ref{lem:2.5}. Note that the matrix $(m_{ij})_{1 \leq i, j \leq 2}$ is unimodular by Lemma \ref{lem:2.1} and Theorem \ref{thm:2.3} (1). We have $\mu_i=m_{1i}+m_{2i}$ for $i=1, 2$ by the condition $(\dagger)$. 

(1) We may assume that $k=1$. Then $(m_{12}, m_{22})=(1,0), (0,1)$ by Lemma \ref{lem:2.5}. If the former holds, then $\mathrm{det}(m_{ij})=-m_{21}=-1$ and hence $H_1 \sim D_2$. If the latter holds, then $H_1 \sim D_1$ by the same method. 

(2) If $(\mu_1, \mu_2) \neq (2,2)$ in addition, then $V \cong \P^1 \times \P^2$ by \cite{m-m83}. Suppose that $(\mu_1, \mu_2) = (2,2)$. Then $(m_{1k}, m_{2k})=(2,0), (0, 2)$ or $(1,1)$ for $k=1,2$. If one of the former two cases occurs, then $\mathrm{det}(m_{ij})$ is even, a contradiction. Hence $m_{ij}=1$ for all $i,j$, which implies $\mathrm{det}(m_{ij})=0$, a contradiction.
\end{proof}

The following lemma is useful.

\begin{lem}\label{lem:euler}
It holds that $\eu(D_1 \cap D_2)= \eu(D_1) +\eu(D_2)+B_3(V)-5.$
\end{lem}

\begin{proof}
By the Mayer-Vietoris exact sequence, we have $\eu(D_1 \cap D_2)= \eu(D_1) +\eu(D_2)-\eu(D_1 \cup D_2)$. Since $U$ is contractible, we have $\eu(D_1 \cup D_2) = \eu(V) -1=5-B_3(V)$. 
\end{proof}

\section{Construction of examples of type (A)}\label{sec:constA}
In this section, we construct examples of triplets with $(\dagger)$ of type (A1)--(A4).

\subsection{Affine modifications}\label{sec:affmod}
We review results about affine modifications, which is a technique to construct new contractible affine varieties from well-known ones.

\begin{defi} [cf. \cite{k-z99, ki05}]
Let $R$ be an affine domain and let $Z \coloneqq \Spec(R)$.
Let $I \subset R$ be an ideal and let $f$ be an element of $I$. We denote by $D \coloneqq V_Z(f ) \supset
C \coloneqq V_Z(I )$ the subschemes in $Z$ defined by $f$ and $I$, respectively. Then {\it the affine
modification of $Z$ with a locus $(C \subset D)$} is the affine variety $Z'$ with a coordinate ring $R'\coloneqq \Gamma(\mathcal{O}_{Z'}) = R[I/f ]$. It is clear that $R \subset R' \subset Q(R)$, where $Q(R)$ is the
quotient field of $R$. The canonical inclusion $R \subset R'$ induces a birational morphism
$\tau \colon Z' \rightarrow Z$. 
We often call this morphism $\tau$ itself as an affine modification with a
locus $(C \subset D)$.
If $f$ is not a unit in $R'$, it defines a divisor $E \coloneqq \div(f)_{Z'}$ on $Z'$ and we have an isomorphism $\tau|_{Z' \setminus E} \colon Z' \setminus E \rightarrow Z \setminus D$. Then $E$ is said to be an {\it exceptional divisor} of the morphism $\tau$. 
\end{defi}

\begin{thm}[cf. {\cite[Theorem 3.1, Corollary 3.1]{k-z99}}]\label{thm:KZ}
Let $Z$ be a contractible affine variety.
Let $\tau \colon Z' \rightarrow Z$ be an affine modification with a locus $(C \subset D)$. Let $E$ be the exceptional divisor of $\tau$. Suppose that 
\begin{enumerate}
\item[\textup{(I)}] $Z$ and $Z'$ are complex manifolds and $D$ and $E$ are topological manifolds with finite decompositions into irreducible components $D= \sum_{i=1}^{n}D_i$ and $E=\sum_{i=1}^{n}E_i$ respectively such that $E_i= \tau^*(D_i)$ for $ i=1, \dots , n$, and
\item[\textup{(I\hspace{-.1em}I)}] $\tau(E_i) \not \subset \Sing\,D_i$ for $i=1, \dots , n.$
\end{enumerate}

Then the following are equivalent:
\begin{enumerate}
\item[\textup{(1)}] $Z'$ is contractible.
\item[\textup{(2)}] $(\tau|_{E_i})_*\colon H_*(E_i, \Z) \rightarrow H_*(D_i, \Z)$ is an isomorphism for all $i=1, \dots , n$.
\end{enumerate}
\end{thm}

Suppose that a Fano 3-fold $V$ is the blow-up of a variety $W$ along a smooth curve and $W$ contains a contractible affine 3-fold as an open set. Then we use the following lemma to find contractible affine 3-folds in $V$, which is a corollary of Theorem \ref{thm:KZ}.

\begin{lem}\label{lem:const1}
Let $W$ be a smooth projective 3-fold. Let $(S_1, S_2, C)$ be a triplet of subvarieties of $W$, where $S_1$ and $S_2$ are closed surfaces and $C$ is a closed smooth curve. Assume that the following holds:
\begin{enumerate}
\item[\textup{(I)}] $C \not\subset S_1, C \subset S_2$, and $C \not\subset \Sing\,S_2$.
\item[\textup{(I\hspace{-.1em}I)}] $W \setminus S_1$ is contractible and affine.
\item[\textup{(I\hspace{-.1em}I\hspace{-.1em}I)}] $S_2 \setminus (S_1 \cap S_2)$ is smooth.
\end{enumerate}
Let $\varphi\colon V \rightarrow W$ be the blow-up along $C$. Then the following are equivalent:
\begin{enumerate}
\item[\textup{(1)}] the open subvariety $U\coloneqq V \setminus (\varphi^{-1}_*(S_{1}) \cup \varphi^{-1}_*(S_{2}))$ is contractible and affine.
\item[\textup{(2)}] the inclusion $\iota\colon C \setminus (S_1 \cap C) \hookrightarrow S_2 \setminus (S_1 \cap S_2)$ induces an isomorphism $H_*( C \setminus (S_1 \cap C), \Z) \rightarrow H_*(S_2 \setminus (S_1 \cap S_2), \Z)$.
\end{enumerate}
\end{lem}

\begin{proof}
Applying Theorem \ref{thm:KZ} by setting $\tau=\varphi|_U, (Z', E)=(U, E_{\varphi} \cap U)$ and $(Z, D)=(W \setminus S_1, S_2 \setminus (S_1 \cap S_2))$, we have the assertion.
\end{proof}

In the forthcoming article \cite{n2}, we investigate the isomorphism class of $U$ as above when $V$ is of type (A1) or (A2), that is, when $W \cong \P^3$.

\begin{nota}
For closed surfaces $S_1$ and $S_2$ in a smooth projective 3-fold $W$, we define a condition (I') as
\begin{enumerate}
\item[\textup{(I')}] $S_1+S_2 \sim -K_W$.
\end{enumerate}
\end{nota}

The following lemma shows us that the condition (I) of Lemma \ref{lem:const1} is necessary to construct triplets $(V, D_1 \cup D_2, U)$ with $(\dagger)$ of type (A).

\begin{lem}\label{lem:condiI}
Let $(V, D_1 \cup D_2, U)$ be a triplet as in Problem \ref{pr:main}. Assume that $V$ is given by the blow-up $\varphi\colon V \rightarrow W$ of a Fano variety $W$ along a smooth curve $C$. Let $S_i\coloneqq \varphi(D_i)$ for $i=1,2$.
Then $(\dagger)$ holds for $D_1$ and $D_2$ if and only if \textup{(I)} and \textup{(I')} hold for $S_1$ and $S_2$ up to permutation.
\end{lem}

\begin{proof}
The “if” part is obvious. Suppose that $(\dagger)$ holds for the triplet. Since the length of $\varphi$ is one, we may assume $D_1 \sim \mathcal{O}_W(1)$ and $D_2 \sim \mathcal{O}_W(r-1)-E_{\varphi}$ by Proposition \ref{prop:2.6} and the condition $(\dagger)$, where $r$ is the Fano index of $W$. Hence $S_1=\varphi(D_1)$ and $S_2=\varphi(D_2)$ satisfy (I) and (I'). 
\end{proof}

\begin{rem}\label{rem:123}
We use the notation as in Lemma \ref{lem:condiI}. In \S\ref{sec:condiIII}, we prove the following relations among $(\dagger)$ and (I)--(I\hspace{-.1em}I\hspace{-.1em}I) of Lemma \ref{lem:const1}.
\begin{itemize}
\item Suppose $(\dagger)$ holds. Then (I\hspace{-.1em}I) holds if and only if $B_2(S_1)$ is smallest possible among hyperplane sections of $W$.
\item Suppose $(\dagger)$ and (I\hspace{-.1em}I) hold, and $S_2$ is normal and rational. Then (I\hspace{-.1em}I\hspace{-.1em}I) holds.
\end{itemize}
Hence (I\hspace{-.1em}I) holds when $(V, D_1 \cup D_2, U)$ is a triplet with $(\dagger)$ of type (A1) or (A2) since $S_1 \cong \P^2$. In the forthcoming article \cite{n2}, we show that the second relation still holds without the rationality of $S_2$.
\end{rem}

Now we choose four varieties $C, S_1, S_2, W$ such that the pair (W, C) satisfies one of the statements (A1)--(A4) and the triplet $(S_1, S_2, C)$ satisfies (I)--(I\hspace{-.1em}I\hspace{-.1em}I), (I') and (1) of Lemma \ref{lem:const1}. 
Then we complete the proof of Theorem 1.1 (2) for all deformation equivalent classes as in (A) of Theorem \ref{thm:main} by Lemmas \ref{lem:const1} and \ref{lem:condiI} in \S\ref{sec:constA12}--\ref{sec:constA4}.

Note that triplets with $(\dagger)$ of type (A) are not necessarily constructed from triplets of subvarieties which satisfy all the condition of Lemma \ref{lem:const1}. In fact, we shall construct a triplet with $(\dagger)$ of type (A3) from a triplet of subvarieties which does not satisfy the condition (I\hspace{-.1em}I) in \S\ref{sec:counterexA}.

\subsection{The types (A1) and (A2)}\label{sec:constA12}
We construct triplets $(V, D_1 \cup D_2, U)$ with $(\dagger)$ of types (A1) and (A2).
\begin{prop} \label{prop:const1}
There exists a triplet $(V, D_1 \cup D_2, U)$ with $(\dagger)$ such that $V$ is a blow-up of $\P^3$ along an elliptic curve of degree $d$ for $d=3,4$.
\end{prop}

\begin{proof}
Let $S_2$ be the cone over a plane elliptic curve $e$, and $v$ the vertex of $S_2$. Let $\pi\colon \wt{S}_2 \rightarrow S_2$ be the minimal resolution. Then $\wt{S}_2$ is a $\P^1$-bundle over $e$, which corresponds to a decomposable vector bundle of rank 2 and degree 3. Let $C_0$ (resp.\,$f$) be the minimal section (resp.\,a ruling) of $\wt{S}_2$. On $\wt{S}_2$, take a smooth member $\wt{C}_d \in \left|C_0+df \right|$ and let $C_d\coloneqq \pi(\wt{C}_d)$, which is an elliptic curve of degree $d$ for $d=3,4$.

For $d=3$, let $S_1$ be a hyperplane which contains $v$. Then the triplet $(S_1, S_2, C_3)$ satisfies the conditions (I\hspace{-.1em}I), (I\hspace{-.1em}I\hspace{-.1em}I) and (I'). Since the intersection $S_1 \cap S_2$ is a sum of rulings of $S_2$, the surface $S_1$ does not contain $C_3$. Hence the condition (I) holds for the triplet $(S_1, S_2, C_3)$. It also holds that $S_2 \setminus (S_1 \cap S_2)$ is isomorphic to an $\A^1$-bundle over $C_3 \setminus (S_1 \cap C_3)$. Hence the triplet $(S_1, S_2, C_3)$ satisfies the condition (1).

For $d=4$, let $\wt{l}$ be the ruling of $\wt{S}_2$ such that $\wt{l} \cap \wt{C}_4=\wt{l} \cap C_0$, and $l\coloneqq \pi (\wt{l}).$ Let $S_1$ be a hypersurface which contains $l$. Then $S_1$ contains $v$ and hence the triplet $(S_1, S_2, C_4)$ satisfies the conditions (I)--(I\hspace{-.1em}I\hspace{-.1em}I) and (I') for the same reason as the case that $d=3$. By the choice of $l$, it also holds that $S_2 \setminus (S_1 \cap S_2)$ is isomorphic to an $\A^1$-bundle over $C_4 \setminus (S_1 \cap C_4)$. Hence the condition (1) holds for the triplet $(S_1, S_2, C_3)$.

As a result, the triplet $(S_1, S_2, C_d)$ satisfies the conditions (I)--(I\hspace{-.1em}I\hspace{-.1em}I), (I') and (1) for $d=3,4$.
\end{proof} 

\begin{prop}
There exists a triplet $(V, D_1 \cup D_2, U)$ with $(\dagger)$ such that $V$ is a blow-up of $\P^3$ along a smooth rational curve of degree $d$ for $d=1,2,3, 4$.
\end{prop}

\begin{proof}
We define $S_1$ and $S_2$ in $\P^3_{[x:y:z:t]}$ as
\[
\begin{aligned}
S_1&\coloneqq \{x=0\}\\
S_2&\coloneqq \{x^2z+y^3+xyt=0\}.
\end{aligned}
\]
Since $S_1$ is a plane and $S_2$ is a non-normal irreducible cubic surface, they satisfy the conditions (I\hspace{-.1em}I) and (I'). We write the conductor locus of $S_2$ as $E_{S_2}$. Then we have 
\[
E_{S_2}=\{x=y=0\}=S_1 \cap S_2,
\]
which implies that the condition (I\hspace{-.1em}I\hspace{-.1em}I) holds for $S_1$ and $S_2$.

Let $\psi \colon X \rightarrow \P^3$ be the blow-up along $E_{S_2}$. Then it is easy to see that $\psi^{-1}_*(S_2)$ is smooth and hence $\psi|_{S_2}$ is the normalization of $S_2$.

Since $S_2$ contains three lines $E_{S_2}, \{y=z=0\}$ and $\{x=y, x+z+t=0\}$ which are not concurrent, $S_2$ is not a cone. Hence $S_2$ belongs to the class $(C)$ of \cite{a-f83}. In particular, $\psi^{-1}_*(S_2)$ is isomorphic to $\F_1$ and $E_{\psi}|_{\psi^{-1}_*(S_2)} \sim \Sigma_1 +f_1$. We note that $E_{\psi}|_{\psi^{-1}_*(S_2)}$ is reducible since $S_2 \setminus E_{S_2} \cong \A^2$. 

On $\psi^{-1}_*(S_2) \cong \F_1$, we take general smooth members 
\[
\wt{C}_1 \in \left|f_1\right|, \wt{C}_2 \in \left|\Sigma_1+f_1\right|, \wt{C}_3 \in \left|\Sigma_1+2f_1\right|\textup{ and } \wt{C}_4 \in \left|2\Sigma_1+2f_1\right|
\]
 such that $\wt{C}_d \cap {E}_{\psi}$ consists of a single point for $d=1,2,3,4$.
Since $\Sigma_1 +2f_1$ is very ample, we can choose $\wt{C}_3$ such that its tangent line of $\wt{C}_3$ at $q\coloneqq \wt{C}_3 \cap {E}_{\psi}$ is different from the ruling of $E_{\psi}$ containing $q$. 

Fix $d \in \{1,2,3,4\}$. By construction, the curve
\[
C_d\coloneqq \psi(\wt{C}_d)
\]
is smooth and rational of degree $d$. Since $\wt{C}_d \cap {E}_{\psi}$ consists of a single point, so does $C_d \cap E_{S_2}=C_d \cap S_1$. Hence the triplet $(S_1, S_2, C_d)$ satisfies the condition (I). As $C_d \setminus (S_1 \cap C_d) \cong \A^1$ and $S_2 \setminus (S_1 \cap S_2) \cong \A^2$, the triplet satisfies the condition (1). Hence the triplet $(S_1, S_2, C_d)$ satisfies the conditions (I)--(I\hspace{-.1em}I\hspace{-.1em}I), (I') and (1).
\end{proof}

\subsection{The type (A3)}\label{sec:constA3}
We construct triplets $(V, D_1 \cup D_2, U)$ with $(\dagger)$ of type (A3).

\begin{prop}\label{prop:const3}
There exists a triplet $(V, D_1 \cup D_2, U)$ with $(\dagger)$ such that $V$ is a blow-up of $\Q^3$ along a smooth rational curve of degree $d$ for $d=1,3,4$.
\end{prop}

\begin{proof}
First, we construct hypersurfaces $S_1$ and $S_2$ in $\Q^3$ which compose desired triplets as in Lemma \ref{lem:const1}. Let $l$ be a line in $\Q^3$ and $\psi_1\colon X \rightarrow \Q^3$ the blow-up of $\Q^3$ along $l$. Then $X$ is a Fano 3-fold of No.31 in \cite[Table 2]{m-m81} and there is the extremal contraction $\psi_2\colon X \rightarrow \P^2$ different from $\psi_1$ and it is a $\P^1$-bundle (see \cite[\S I\hspace{-.1em}I\hspace{-.1em}I-3]{ma95}). An easy computation shows that 
\begin{equation}\label{eq:const2}
\psi_2^* \mathcal{O}_{\P^2}(1) \sim \psi_1^* \mathcal{O}_{\Q^3} (1) - E_{\psi_1}.
\end{equation}
We also obtain that $N_l \Q^3=\mathcal{O}_l \oplus \mathcal{O}_l(1)$. Hence $E_{\psi_1} \cong \F_1$ and $\psi_2|_{E_{\psi_1}} \colon E_{\psi_1} \rightarrow \P^2$ is a blow-up at a point. Let $p$ be a center of blow-up $\psi_2|_{E_{\psi_1}}$. Let $r_1 \in \left| \mathcal{O}_{\P^2}(1) \otimes \mathcal{I}_p \right|$ and $r_2 \in \left| \mathcal{O}_{\P^2}(2) \otimes \mathcal{I}_p \right|$ be smooth curves such that $r_1 \cap r_2=\{p\}$. We take $\wt{S}_i$ and $S_i$ for $i=1, 2$ as
\[
\wt{S}_i\coloneqq \psi_2^{-1}(r_i) \textup{ and } S_i\coloneqq \psi_1(\wt{S}_i). 
\] 
Then the pair $(S_1, S_2)$ satisfies the condition (I'). 
\begin{claim}
It holds that $S_1 \cong \Q^2_0$ and $S_2$ is a non-normal Gorenstein del Pezzo surface of class (C) in \cite{a-f83}.
\end{claim}

\begin{claimproof}\label{cl:QF_2}
By (\ref{eq:const2}), we have $S_i \in \left| \mathcal{O}_{\Q^3}(i) \otimes \mathcal{I}_{l}^i\right|$ for $i=1,2$. Hence it suffices to show that $\wt{S}_1 \cong \F_2$ to obtain the first assertion. We also note that $S_2$ is a non-normal Gorenstein del Pezzo surface of degree $4$. By Lemma \ref{lem:cone} and \cite{a-f83}, $S_2$ belongs to 
\[
\left\{
\begin{array}{ll}
\textup{the class (C) of \cite{a-f83}}&\textup{if } \wt{S}_2 \cong \F_2 \textup{ and}\\
\textup{the class (D) of \cite{a-f83}}&\textup{if } \wt{S}_2 \cong \F_0. 
\end{array}
\right.
\]
Hence it suffices to show that $\wt{S}_2 \cong \F_2$ to obtain the second assertion. Note that $\psi_1(E_{\psi_1} \cap \wt{S}_2)$ is the conductor locus $E_{S_2}=l$ of $S_2$. 

Fix $i \in \{1,2\}$. Since $N_{E_{\psi_1}}X=\mathcal{O}_{E_{\psi_1}}(-\Sigma_1)$ and $\wt{S}_i|_{E_{\psi_1}} \sim i(\Sigma_1 + f_1)$, we have
\begin{equation}\label{eq:F_2}
(E_{\psi_1}^2 \cdot \wt{S}_i)=-i(\Sigma_1 \cdot \Sigma_1+f_1)_{E_{\psi_1}}=0.
\end{equation}
By the choice of $r_i$, the intersection $\wt{S}_i|_{E_{\psi_1}}$ is decomposed to two curves $\Sigma_1$ and $C \sim (i-1)\Sigma_1 + if_1$ on $E_{\psi_1}$. Note that $\Sigma_1$ and $C$ are a ruling and a section on $\wt{S}_i$ respectively. By (\ref{eq:F_2}), we have 
\[
0=(\Sigma_1+C)^2=\Sigma_1^2+2(\Sigma_1 \cdot C)+C^2=2+C^2 \textup{ on } \wt{S}_i.
\]
Hence $C^2=-2$ and $S_i \cong \F_2$. 
\hfill $\blacksquare$
\end{claimproof}

Hence $\Q^3 \setminus S_1 \cong \A^3$ by \cite[Theorem A]{fur93} and the condition (I\hspace{-.1em}I) holds for $S_1$ and $S_2$.
By the choice of $r_1$ and $r_2$, 
we have $S_1 \cap S_2=\psi_1(\wt{S}_1 \cap \wt{S}_2)=l$. 
Since 
\[
(\psi_1|_{\wt{S}_2})^{-1}(l)= E_{\psi_1}|_{\wt{S}_2}=\Sigma_2+C
\]
for some ruling $C$ of $\wt{S}_2 \cong \F_2$, we have $S_2 \setminus (S_1 \cap S_2) \cong \F_2 \setminus (\Sigma_2 \cup C) \cong \A^2$. Hence the condition (I\hspace{-.1em}I\hspace{-.1em}I) holds for $S_1$ and $S_2$.

Next we construct desired centers of blow-ups. On $\wt{S}_2 \cong \F_2$, we take general smooth members 
\[
\wt{C}_1 \in \left| f_2 \right|, \wt{C}_3 \in \left| \Sigma_2 +2f_2 \right|\textup{ and } \wt{C}_4 \in \left| \Sigma_2 +3f_2 \right|.
\]
Since $\Sigma_2 +3f_2$ is very ample on $\F_2$, we can choose $\wt{C}_4$ such that $\wt{C}_4$ contains $q\coloneqq \Sigma_2 \cap C$ and its tangent line at $q$ is different from that of the ruling of $E_{\psi_1} \cong \F_1$ containing $q$.

Fix $d \in\{1, 3, 4\}$. By construction, the curve
\[
C_d\coloneqq \psi_1(\wt{C}_d)
\]
is smooth and rational of degree $d$. Since $\wt{C}_d \cap E_{\psi_1}= \wt{C}_d \cap (\Sigma_2 \cup C)$ consists of a single point, so does $C_d \cap l=C_d \cap S_1$. Hence the triplet $(S_1, S_2, C_d)$ satisfies the condition (I). 
As $C_d \setminus (S_1 \cap C_d) \cong \A^1$ and $S_2 \setminus (S_1 \cap S_2) \cong \A^2$, the triplet also satisfies the condition (1). Hence the triplet $(S_1, S_2, C_d)$ satisfies the conditions (I)--(I\hspace{-.1em}I\hspace{-.1em}I), (I') and (1).
\end{proof}

\begin{prop}\label{prop:const3.5}
There exists a triplet $(V, D_1 \cup D_2, U)$ with $(\dagger)$ such that $V$ is a blow-up of $\Q^3$ along a conic.
\end{prop}

\begin{proof}
We use the notation $l, \psi_1, \psi_2, p, r_1, \wt{S}_1$ and $S_1$ as in Proposition \ref{prop:const3}. We take a smooth conic $r_2 \subset \P^2$ such that $r_1 \cap r_2$ consists of a single point which is different from $p$. We take 
\[
\wt{S}_2\coloneqq \psi_2^{-1}(r_2)\textup{ and }S_2\coloneqq \psi_1(\wt{S}_2).
\]
Then the pair $(S_1, S_2)$ satisfies the conditions (I\hspace{-.1em}I) and (I'). 

A computation as in the proof of Claim \ref{cl:QF_2} shows that $S_2$ belongs to the class (D) of \cite{a-f83}. More precisely, we have $\wt{S}_2 \cong \F_0$, $E_{\psi_1}|_{\wt{S}_2} \sim \Sigma_0$ and 
\[
(\psi_1|_{\wt{S}_2})^{-1}(S_1 \cap S_2)=E_{\psi_1}|_{\wt{S}_2}+C
\]
for some ruling $C \in \left| f_0 \right|$ on $\wt{S}_2 \cong \F_0$. Hence $S_2 \setminus (S_1 \cap S_2) \cong \wt{S}_2 \setminus (E_{\psi_1} \cup C) \cong \A^2$ and the condition (I\hspace{-.1em}I\hspace{-.1em}I) holds for $S_1$ and $S_2$. 

We take curves 
\[
\wt{C}_2 \in \left| \Sigma_0 \right| \textup{ and } C_2\coloneqq \psi_1(\wt{C}_2)
\]
such that $\wt{C}_2 \neq E_{\psi_1}|_{\wt{S}_2}$. Then $C_2$ is a smooth conic. Since $\wt{C}_2 \cap (E_{\psi_1}|_{\wt{S}_2} \cup C)$ consists of a single point, so does $C_2 \cap S_1$. Hence the triplet $(S_1, S_2, C_2)$ satisfies the conditions (I) and (1) as in the proof of Proposition \ref{prop:const3}.
\end{proof}

\subsection{The type (A4)}\label{sec:constA4}
We construct triplets $(V, D_1 \cup D_2, U)$ with $(\dagger)$ of type (A4).

\begin{prop}\label{prop:const4}
There exists a triplet $(V, D_1 \cup D_2, U)$ with $(\dagger)$ such that $V$ is a blow-up of $V_5$ along a line.
\end{prop}

\begin{proof}
We take hyperplane sections
\[
S_1=H_5^0 \textup{ and } S_2=H_5^{\infty}
\]
as in \cite{fur00}. Then the pair $(S_1, S_2)$ satisfies the condition (I'). Note that $S_2$ is non-normal. By \cite[Lemma 7]{fur00}, the normalization $\ol{S}_2$ of $S_2$ is isomorphic to $\F_3$. Hence $S_2$ belongs to the class (C) of \cite{a-f83} and the conductor locus $\ol{E}_{S_2}$ of $\ol{S}_2$ is decomposed as $\Sigma_3 \cup f_3$. Since the intersection $S_1 \cap S_2$ is the conductor locus of $S_2$ by \cite[Lemma 11]{fur00}, we have $S_2 \setminus (S_1 \cap S_2) \cong \F_3 \setminus (\Sigma_3 \cup f_3) \cong \A^2$. Hence the condition (I\hspace{-.1em}I\hspace{-.1em}I) holds for $S_1$ and $S_2$. The pair also satisfies the condition (I\hspace{-.1em}I) since $V_5 \setminus S_1 \cong \A^3$ by \cite[Lemma12]{fur00}.

We take $\wt{C}_1 \subset \ol{S}_2$ and $C_1 \subset S_2$ as
\[
\wt{C}_1 \in \left| f_3 \right| \textup{ with } \wt{C}_1 \not \subset \ol{E}_{S_2} \textup{ and } C_1\coloneqq \sigma_{S_2}(\wt{C}_1),
\]
where $\sigma_{S_2} \colon \ol{S}_2 \rightarrow S_2$ is the normalization. Then $C_1$ is a line in $S_2$ different from $E_{S_2}$ and hence the triplet $(S_1, S_2, C_1)$ satisfies the condition (I). As $C_1 \setminus (S_1 \cap C_1) \cong \A^1$ and $S_2 \setminus (S_1 \cap S_2) \cong \A^2$, the triplet $(S_1, S_2, C_1)$ also satisfies the condition (1). Hence the triplet $(S_1, S_2, C_1)$ satisfies the conditions (I)--(I\hspace{-.1em}I\hspace{-.1em}I), (I') and (1).
\end{proof}

\begin{prop}\label{prop:const4.5}
There exists a triplet $(V, D_1 \cup D_2, U)$ with $(\dagger)$ such that $V$ is a blow-up of $V_5$ along a smooth rational curve of degree $d$ for $d=2,3$.
\end{prop}

\begin{proof}
We take hyperplane sections
\[
S_1=H_5^{\infty} \textup{ and } S_2=H_5^{0}
\]
as in \cite{fur00}. Then the pair $(S_1, S_2)$ satisfies the condition (I'). By \cite[Lemma 12]{fur00}, we have $V_5 \setminus S_1 \cong \A^3$ and $S_2 \setminus (S_1 \cap S_2) \cong \A^2$. Hence $S_1$ and $S_2$ satisfy the conditions (I\hspace{-.1em}I) and (I\hspace{-.1em}I\hspace{-.1em}I). 

Fix $d \in \{2, 3\}$ and take $C_d$ as a smooth rational curve of degree $d$ in $S_2$ as in \cite[Lemma 2.3]{ki05}. Since $S_1 \cap S_2$ is a line, the triplet $(S_1, S_2, C_d)$ satisfies the condition (I). As $C_d \setminus (S_1 \cap C_d) \cong \A^1$ and $S_2 \setminus (S_1 \cap S_2) \cong \A^2$ by construction, the triplet $(S_1, S_2, C_d)$ satisfies the condition (1).
\end{proof}

\subsection{Another example of type (A3)}\label{sec:counterexA}

We give an example of triplets with $(\dagger)$ of type (A3) such that the image of its boundary divisors by the blow-up morphism does not satisfy the condition (I\hspace{-.1em}I) of Lemma \ref{lem:const1}.

\begin{ex}\label{ex:counterex}
Take $S_1, S_2$ and $C$ in $\Q^3=\{X_1^2+X_0X_4+X_2X_3=0\} \subset \P^4_{[X_0: \dots :X_4]}$ as
\[
\begin{aligned}
S_1&\coloneqq \{X_1^2+X_0X_4+X_2X_3=0=X_1=0\} \cong \P^1 \times \P^1, \\
S_2&\coloneqq \{X_1^2+X_0X_4+X_2X_3=0=X_0^2+X_1X_2=0\}, \textup{ and }\\
C&\coloneqq \{X_0=X_1=-X_2=X_3-X_4\}.
\end{aligned}
\]
Then $S_1 \cong\P^1 \times \P^1$. Since $\Q^3 \setminus S_1 \cong \mathrm{SL}(2;\C)$, the triplet $(S_1, S_2, C)$ does not satisfy the condition (I\hspace{-.1em}I).

Let $g \colon V \rightarrow \Q^3$ be the blow-up of $\Q^3$ along the line $C$ and $D_i\coloneqq g^{-1}_*(S_i)$ for $i=1,2$. Then the open set $U\coloneqq V \setminus (D_1 \cup D_2)$ is the affine modification of 
$
\{1+x_0x_4+x_2x_3=0\}
$
with the locus 
\[
\left(
\begin{array}{l}
(\{1+x_0x_4+x_2x_3=0, x_0^2+x_2=0, x_0=1\} \\
\subset \{1+x_0x_4+x_2x_3=0, x_0^2+x_2=0\}
\end{array}
\right)
\]
in $\A^4_{(x_0, x_2, x_3, x_4)}$, which is isomorphic to
\[
\begin{aligned}
&\{1+x_0x_4+x_2x_3=0, (x_0^2+x_2)w=x_0-1\}\\
\cong &\{1+x_0x_4+(x_2-x_0^2)x_3=0, x_0=1+x_2w\}\\
\cong &\{1+x_0(x_4-x_0x_3)+x_2x_3=0, x_0=1+x_2w\}\\
\cong &\{1+x_0x_4+x_2x_3=0, x_0=1+x_2w\}\\
\cong &\{1+x_4+x_2(x_3+x_4w)=0, x_0=0\}\\
\cong &\{x_4=x_0=0\} \cong \A^3
\end{aligned}
\]
in $\A^5_{(x_0,x_2,x_3,x_4,w)}$. As $S_1$ and $S_2$ satisfy the conditions (I) and (I'), we have $K_V+D_1+D_2=0.$
Hence $(V, D_1 \cup D_2, U)$ is a triplet with $(\dagger)$ of type (A3).
\end{ex}

\section{Construction of examples of type (B)}\label{sec:constB}

In this section, we construct examples of triplets with $(\dagger)$ of type (B1)--(B3).

\subsection{The characterization of $\A^3$}\label{sec:Kaliman}
We review the characterization of $\A^3$ by Sh.\,Kaliman.

\begin{thm}[cf. \cite{ka02}]\label{thm:2.26}
Let $X$ be an affine 3-fold such that
\begin{enumerate}
\item[\textup{(0)}] $\Pic X=0$ and all invertible functions on $X$ are constants;
\item[\textup{(1)}] The Euler characteristic of $X$ is $\eu(X)=1$;
\item[\textup{(2)}] There exists a Zariski open subset $Z$ of $X$ and a morphism $p\colon Z \rightarrow r$ onto a curve $r$ whose fibers are isomorphic to $\A^2$;
\item[\textup{(3)}] Each irreducible component of $X \setminus Z$ has the trivial Picard group.
\end{enumerate}

Then $X$ is isomorphic to $\A^3$. 
\end{thm}

\subsection{The type (B3)}\label{sec:constB3}

With Theorem \ref{thm:2.26}, we can construct a triplet $(V, D_1 \cup D_2, U)$ with $(\dagger)$ of type (B3).

\begin{prop}\label{prop:const5}
Let $x_1, x_2, x_3$ and $y$ be coordinates of $\P\coloneqq \P(1,1,1,2)$ of degree $1,1,1$ and $2$ respectively. Define $S_1$ and $S_2 \subset \P$ as
\[
\begin{aligned}
S_1&\coloneqq \{y=0\} \textup{ and }\\
S_2&\coloneqq \{yx_1+x_2x_3(x_2+x_3)=0\}.
\end{aligned}
\]
Let $\varphi\colon V \rightarrow \P\coloneqq \P(1,1,1,2)$ be the blow-up of $\P$ at the vertex and $D_i\coloneqq \varphi^{-1}_*(S_i)$ for $i=1,2$. Then the following holds:
\begin{enumerate}
\item[\textup{(1)}] $K_V+D_1+D_2=0$.
\item[\textup{(2)}] $D_1$ and $D_2$ are $\Z$-basis of $\Pic V$.
\item[\textup{(3)}] $E_{\varphi} \setminus (D_2 \cap E_{\varphi}) \cong \A^2$.
\item[\textup{(4)}] $\P \setminus (S_1 \cup S_2) \cong \A^2 \times \C^*$.
\item[\textup{(5)}] $V \setminus (D_1 \cup D_2) \cong \A^3$.
\end{enumerate}
In particular, $(V, D_1 \cup D_2, \A^3)$ is a triplet with $(\dagger)$ of type \textup{(B3)}.
\end{prop}

\begin{proof}
\noindent(1): Let $H\coloneqq \varphi^*\mathcal{O}_\P(2)$, then $D_1 \sim H$. By the Jacobian criterion, the vertex $p$ of $V$ is the unique singular point of $S_2$ and $(S_2, p)$ is an $A_1$-singularity. Hence $E_{\varphi}|_{D_2}$ is a (-2)-curve. 
We write $D_2 \sim aH+bE_{\varphi}$. Then we have $4a=H^2 \cdot D_2=6$ and $4b=E^2_{\varphi} \cdot D_2=-2$. Hence we have $D_2 \sim \frac{3}{2}H-\frac{1}{2}E_{\varphi}$, which proves (1).

\noindent(2): Let $F$ be the pullback of $\mathcal{O}_{\P^2}(1)$ by the $\P^1$-bundle structure of $V=\P_{\P^2}(\mathcal{O} \oplus \mathcal{O}(2))$. By Theorem \ref{thm:2.3} and (1), we have $D_2 \sim H+F$ and $\Pic V=\Z[H] \oplus \Z[F]=\Z[D_1] \oplus \Z[D_2]$.

\noindent(3): Since $(E_{\varphi} \cdot D_2^2)=1$, the intersection $D_2|_{E_{\varphi}}$ is a line in $E_{\varphi} \cong \P^2$ and hence $E_{\varphi} \setminus (D_2 \cap E_{\varphi}) \cong \A^2$.

\noindent(4): Let $\alpha$ be the involution on $\A^3_{(x_1, x_2, x_3)}$ which sends $a$ to $-a$. 
Let $\pi\colon \A^3 \rightarrow \A^3 /\{\mathrm{id}, \alpha\}$ be the quotient morphism.
By regarding $\{y \neq 0\} \subset \P$ as $\A^3 /\{\mathrm{id}, \alpha\}$, we have an isomorphism 
\[
\P \setminus (S_1 \cup S_2)\cong(\A^3 /\{\mathrm{id}, \alpha\}) \setminus \pi(\{x_1+x_2x_3(x_2+x_3)=0\}).
\] 
Consider the polynomial automorphism $\beta$ of $\A^3$ such that 
\[
\beta(x_1)=x_1-x_2x_3(x_2+x_3), \beta(x_2)=x_2, \beta(x_3)=x_3.
\]
Since $\alpha$ commutes with $\beta$, we have the desired isomorphism 
\[
\begin{aligned}
\P \setminus (S_1 \cup S_2)
\cong&(\A^3 /\{\mathrm{id}, \alpha\}) \setminus \pi(\{\beta(x_1+x_2x_3(x_2+x_3))=0\})\\
\cong&
(\A^3 /\{\mathrm{id}, \alpha\}) \setminus \pi(\{x_1=0\})\\
\cong&
(\A^2_{(x_2, x_3)} \times \C^*_{(x_1)})/\{\mathrm{id}, \alpha\} \\
\cong& \A^2 \times \C^*.
\end{aligned}
\]
\noindent(5): By (1), the compliment $V \setminus (D_1 \cup D_2)$ is affine. By (2) and \cite[Proposition 1.1(1)]{fuj82}, the condition (0) of Theorem \ref{thm:2.26} holds for $V \setminus (D_1 \cup D_2)$. By (3) and (4), we have $\eu(V \setminus (D_1 \cup D_2))=\eu(E_{\varphi} \setminus (D_2 \cap E_{\varphi}))+\eu(\P \setminus (S_1 \cup S_2))=\eu(\A^2)+\eu(\A^2\times \C^*)=1$. Hence the condition (1) of Theorem \ref{thm:2.26} holds for $V \setminus (D_1 \cup D_2)$. Applying Theorem \ref{thm:2.26} by setting $X\coloneqq V \setminus (D_1 \cup D_2)$ and $Z \coloneqq V \setminus (D_1 \cup D_2 \cup E_{\varphi}) \cong \P \setminus (S_1 \cup S_2)$, we obtain (5).
\end{proof}

\subsection{The types (B1) and (B2)}\label{sec:constB12}

To construct examples of triplets with $(\dagger)$ of type (B1) and (B2), we use the following lemma, which is a corollary of Theorem \ref{thm:2.26}.

\begin{lem}\label{lem:condi2}
Let $\varphi\colon V \rightarrow \P^2$ be a $\P^1$-bundle. Let $D_1$ and $D_2 \subset V$ be irreducible and generically birational sections of $\varphi$ which satisfy the following:
\begin{enumerate}
\item[\textup{(0')}] $D_1$ and $D_2$ are $\Z$-basis of $\Pic V$ and $D_1+D_2$ is ample.
\item[\textup{(1')}] There is the unique point $p \in \P^2$ such that $\varphi^{-1}(p) \subset D_1 \cup D_2$.
\item[\textup{(2')}] $\varphi(D_1 \cap D_2) \subset \P^2$ is a line containing $p$.

\end{enumerate}
Then the open subvariety $U\coloneqq V \setminus (D_1 \cup D_2)$ is isomorphic to $\A^3$.
\end{lem}

\begin{proof}
Since $D_1+D_2$ is ample, the variety $U$ is affine. By \cite[Proposition 1.18 (1)]{fuj82}, the condition (0') implies the condition (0) of Theorem \ref{thm:2.26} for $U$.

Let $l\coloneqq \varphi(D_1 \cap D_2)$. Then $V_0\coloneqq V \setminus (\varphi^{-1}(l) \cup D_1)$ is $\A^1$-bundle over $\A^2$ by the condition (1') and \cite[\S4.1]{miy78}. Since $V_0 \cap D_2$ is a section of $V_0$ by (1') and (2'), it is easy to check that
\[
U \cap (V \setminus \varphi^{-1}(l)) \cong V_0 \cap (V \setminus D_2) \cong \A^2 \times \C^*. 
\]
Note that the condition (2') implies that $U \cap \varphi^{-1}(l) \cong \A^2$. Hence $U$ satisfies the conditions (2) and (3) of Theorem \ref{thm:2.26} if we set $Z\coloneqq U \cap (V \setminus \varphi^{-1}(l))$. 

Since $\eu (U)=\eu(Z)+\eu(U \setminus Z)=1$, the condition (1) of Theorem \ref{thm:2.26} is also satisfied. Hence we have the assertion by Theorem \ref{thm:2.26}.
\end{proof}

\begin{rem}\label{rem:counterex}
The conditions as in Lemma \ref{lem:condi2} is not necessary to construct triplets $(V, D_1 \cup D_2, U)$ with $(\dagger)$ of type (B). For example, let $(V, D_1 \cup D_2, U)$ be a triplet as in Proposition \ref{prop:const5}. The intersection $D_1 \cap D_2$ consists of three curves, but the conditions (1') and (2') implies the intersection consists of at most two curves.
\end{rem}

We can now construct examples of triplets with $(\dagger)$ of type (B1) and (B2). 

\begin{prop}\label{prop:const7}
There exist divisors $D_1$ and $D_2$ in a smooth 3-fold $V \subset \P^2 \times \P^2$, which is a divisor of bidegree $(1,2)$, such that $(V, D_1 \cup D_2, \A^3)$ is a triplet with $(\dagger)$ of type \textup{(B1)}.
\end{prop}

\begin{proof}
In $\P^2_{[x_0:x_1:x_2]} \times \P^2_{[y_0:y_1:y_2]}$, take $V$, $D_1$ and $D_2$ as 
\begin{eqnarray*}
V &=&\{x_0y_0^2+x_1y_1^2+x_2q(y_0,y_1,y_2)=0\},\\
D_1&\coloneqq &\{x_0y_0^2+x_1y_1^2+x_2q(y_0,y_1,y_2)=x_0y_1 +x_2(by_0+ay_1)=0\},\\
D_2&\coloneqq &\{x_0y_0^2+x_1y_1^2+x_2q(y_0,y_1,y_2)=x_2=0\}
\end{eqnarray*}
with $a \in \C$ and $b \in \C^*$ and where $q$ is a quadric form with $q([0:0:1]) \neq 0$. It is easy to check the smoothness of $V$ by the Jacobian criterion. Let 
$\varphi\coloneqq \mathrm{pr}_2\colon V \rightarrow \P^2$, 
which is a $\P^1$-bundle. We note that $\Pic V$ is generated by 
$H_1\coloneqq \mathrm{pr}_1^* \mathcal{O}_{\P^2}(1)$
 and 
$H_2\coloneqq \mathrm{pr}_2^* \mathcal{O}_{\P^2}(1)$
 by the Grothendieck-Lefschetz theorem. 
Since $D_1 \sim H_1+H_2$ and $D_2 \sim H_1$, $D_1$ and $D_2$ satisfy the condition (0') of Lemma \ref{lem:condi2} and $(\dagger)$. An easy computation shows that $D_1 \cap D_2=\{x_0=y_0=y_1=0\} \cup \{x_0=x_2=y_1=0\}$ and the condition (2') holds for $D_1$ and $D_2$. 

For homogeneous elements $g_0, g_1$ and $g_2 \in \C[y_0, y_1, y_2]$, a divisor $\{f= x_0g_0$ $+x_1g_1+x_2g_2=0 \}$ in $V$ contains $\varphi^{-1}([a:b:c])$ only if the rank of
\[
\left(
\begin{array}{ccc}
a^2 & b^2 & q(a,b,c)\\
g_0(a,b,c)& g_1(a,b,c) & g_2(a,b,c)
\end{array}
\right)
\]
is smaller than two. Hence we can check that $D_1$ and $D_2$ satisfy the condition (1') with $p\coloneqq [0:0:1]$.
\end{proof}

\begin{prop}\label{prop:No24}\label{prop:const6}
There exist divisors $D_1$ and $D_2$ in $\P^1 \times \P^2$ such that $(\P^1 \times \P^2, D_1 \cup D_2, \A^3)$ is a triplet with $(\dagger)$ of type \textup{(B2)}.
\end{prop}

\begin{proof}
In $V=\P^1_{[x_0:x_1]} \times \P^2_{[y_0:y_1:y_2]}$, we take $D_1$ and $D_2$ as 
\begin{eqnarray*}
D_1&\coloneqq &\{x_0y_1 +x_1y_2=0\},\\
D_2&\coloneqq &\{x_0y_1(ay_1+y_2)+x_1(by_1^2+ay_1y_2+y_2^2)=0\}
\end{eqnarray*}
with $a \in \C$ and $b \in \C^*$. 
Since $D_i$ is of bidegree $(1, i)$ for $i=1,2$, the condition (0') and $(\dagger)$ holds for $D_1$ and $D_2$. They also satisfy the condition (1') with $\varphi\coloneqq \mathrm{pr}_2$ and $p\coloneqq [1:0:0]$. This implies that $D_1$ and $D_2$ are irreducible.
An easy computations shows that $D_1 \cap D_2= \{y_1=y_2=0\} \cup \{x_1=y_1=0\}$. Hence $D_1$ and $D_2$ satisfy the condition (2') and we complete the proof.
\end{proof}

Summarizing Propositions \ref{prop:const1}--\ref{prop:const4.5}, \ref{prop:const5}, \ref{prop:const7} and \ref{prop:const6}, we have proved Theorem 1.1 (2).

In the remaining part of this article, we fix a triplet $(V, D_1 \cup D_2, U)$ with $(\dagger)$. We shall prove Theorem \ref{thm:main} (1) to seek a contradiction to Lemma \ref{lem:euler} when $V$ is of type neither (A) or (B). To obtain a contradiction, we compute $\eu(D_1)$, $\eu(D_2)$ and $\eu(D_1 \cap D_2)$. 
From now on, we use Notation \ref{no:Fano}.

\section{Exclusion of imprimitive Fano 3-folds}\label{sec:excimpri}
The aim of this section is to prove Theorem \ref{thm:main} (1) when $V$ is imprimitive. For this reason, we assume that $V$ is the blow-up of a Fano variety $W$ of index $r$ along a smooth curve $C$. We use Notation \ref{no:Fano} and fix $\varphi_1=\Bl_C$. By Proposition \ref{prop:2.6}, we may assume that
\begin{equation}\label{eq:H}
D_1 \sim \varphi^*\mathcal{O}_{W}(1) \textup{ and }D_2 \sim \varphi^*\mathcal{O}_{W}(r-1) - E_{\varphi}.
\end{equation}
\begin{nota} We also use the following notation.
\begin{itemize}
\item $S_i\coloneqq {\varphi_1}_*(D_i) \:\mathrm{for}\: i=1,2.$ 
\item $\tau\colon \wt{S}_2 \rightarrow S_2$: the composition of the normalization and the minimal resolution.
\item $F\coloneqq (S_1 \cap S_2)_{\mathrm{red}}$.
\item $J_1\coloneqq \!\!\{l:$ a curve $ \subset D_1 \cap E_{\varphi_1}\mid \varphi_1(l)$ is a point$\}, N_1\coloneqq \sharp J_1$.
\item $J_2\coloneqq \!\!\{l:$ a curve $ \subset D_2 \cap E_{\varphi_1}\mid \varphi_1(l)$ is a point$\}, N_2\coloneqq \sharp J_2$.
\item $J_{1 \cap 2}\coloneqq \!\!\{l:$ a curve $ \subset (D_1 \cap D_2) \cap E_{\varphi_1} \mid \varphi_1(l)$ is a point$\}, N_{1 \cap 2}\coloneqq \sharp J_{1 \cap 2}$.
\end{itemize}
By (\ref{eq:H}), we have
$
N_1=\sharp(C \cap S_1), N_2=\sharp(C \cap \Sing\, S_2)\textup{ and }N_{1 \cap 2}=\sharp(C \cap S_1 \cap \Sing\, S_2).
$ We
% often 
use Notation \ref{no:dP} when $S_2$ is non-normal.
\end{nota}

\subsection{The image of contraction $\varphi_1 \colon V \rightarrow W_1$}\label{sec:image}

\begin{lem}\label{lem:2.9}
It holds that $\eu(D_i)= \eu(S_i)+N_i$ for $i=1, 2$, and $\eu(D_1 \cap D_2)= \eu(F)+N_{1 \cap 2}$.
\end{lem}

\begin{proof}
We have the first equations as follows:
\[
\begin{aligned}
\eu(D_i)
&=\eu\left(D_i\setminus \bigcup_{l \in J_i}l\right)+\eu\left(\bigcup_{l \in J_i}l\right)
\\
 &=\eu\left(S_i\setminus \bigcup_{l \in J_i}\varphi(l)\right)+N_i \times \eu(\P^1)\\
 &=\eu(S_i)+N_i.
\end{aligned}
\]
The second assertion follows from the same argument.
\end{proof}

\begin{prop}\label{prop:2.10}
We have the following:
\begin{enumerate}
\item[\textup{(1)}] $\eu(F) = \eu(S_1)+\eu(S_2) +B_3(W_1)+2p_a(C)+ N_1 +N_2 - N_{1 \cap 2}-5$.
\item[\textup{(2)}] $N_1 + N_2 -N_{1 \cap 2} \geq 1$ and the equality holds if and only if $(N_1, N_2, N_{1 \cap 2})$ $=$ $(1,0,0)$ or $(1,1,1)$.
\end{enumerate}
\end{prop}

\begin{proof}
(1) The assertion holds by Lemma \ref{lem:euler}, Lemma \ref{lem:2.9} and the equation $B_3(V)=B_3(W_1)+2p_a(C)$. \\
(2) The desired inequality is given by following inequalities: 
\[
N_1\geq 1\textup{ and }\mathrm{max}\{N_1, N_2\} \geq N_{1 \cap 2}.
\]
\end{proof}

\begin{lem}\label{lem:deg}
It holds that $\eu(F) \leq 1 + (S_1^2 \cdot S_2)$.
\end{lem}

\begin{proof}
Since $S_1|_{S_2} \subset S_2$ is ample, its support is connected. Since $S_1 \subset W_1$ is ample, we have $B_2(F) \leq (S_1^2 \cdot S_2)$. Hence $\eu(F) \leq B_0(F) + B_2(F) \leq 1 + (S_1^2 \cdot S_2)$.
\end{proof}

\begin{prop}\label{prop:W_1}
It holds that $W_1 = \P^3, \Q^3 $ or $ V_5$.
\end{prop}

\begin{proof}
By \cite[Proposition 5.12]{m-m83}, it suffices to show that $W_1$ is not del Pezzo 3-fold $V_d$ of degree $1 \leq d \leq 4$. Suppose that $W_1 = V_d$ for some $1 \leq d \leq 4$.

By (\ref{eq:H}), the surfaces $S_1$ and $S_2$ are hypersurfaces of $V_d$. By Proposition \ref{prop:2.10} and Lemma \ref{lem:deg}, we have
\begin{equation}\label{eq:dP}
1+d = 1 + (S_1^2 \cdot S_2) \geq \eu(S_1) + \eu(S_2) +B_3(V_d) -4.
\end{equation}
Together with Theorem \ref{thm:dP}, we have $3+d \geq B_3(V_d)$. Since 
\[
B_3(V_d)=
\left\{
\begin{array}{ll}
42&\mathrm{if}\:d=1\\
20&\mathrm{if}\:d=2\\
10&\mathrm{if}\:d=3\\
4 &\mathrm{if}\:d=4
\end{array}
\right.
\]
by \cite[Table 3.5]{isk80},
we have $d=4$. By Theorem \ref{thm:dP} (4) and Lemma \ref{lem:cone}, we have $\eu(S_i) \geq 3$ for $i=1,2$. Hence (\ref{eq:dP}) is rewritten as $5=1+d \geq 6$, a contradiction. 
\end{proof}

\subsection{On the conditions in Lemma \ref{lem:const1}}\label{sec:condiIII}

Now we can prove the relations stated in Remark \ref{rem:123}. Note that we use only Proposition \ref{prop:5.3} among propositions in this subsection to prove Theorem \ref{thm:main} (1).

\begin{prop}\label{prop:condiII}
The condition \textup{(I\hspace{-.1em}I)} in Lemma \ref{lem:const1} holds for $(S_1, S_2, C)$ if and only if $B_2(S_1)$ is smallest possible among hyperplane sections of $W_1$.
\end{prop}

\begin{proof}
By (\ref{eq:H}) and Proposition \ref{prop:W_1}, the surface $S_1$ is a hyperplane section of $W_1 \cong \P^3$, $\Q^3$ or $V_5$. By the classification of compactifications of contractible affine 3-folds into Fano 3-folds with $B_2=1$ (see \cite[Corollary 2.1]{ki05}), we have the assertion.
\end{proof}

\begin{prop}\label{prop:5.2} Suppose that the condition \textup{(I\hspace{-.1em}I)} in Lemma \ref{lem:const1} holds for $(S_1, S_2)$. We also assume that $S_2$ is normal and rational. Then $S_2 \setminus F$ is smooth.
\end{prop}

\begin{proof} 
By Proposition \ref{prop:W_1}, we have $W_1 \cong \P^3, \Q^3$ or $V_5$. Hence $S_2$ is a Gorenstein del Pezzo surface with $K_{S_2}^2 \geq 3$ by (\ref{eq:H}).
Let $S\coloneqq S_2 \setminus F$ and $\Sing\,S\coloneqq \{p_1, \dots, p_n\}$ $(n \geq 0)$. For all $i$, take an open ball $U_i'\coloneqq \{x \in \C^3 \:|\: d(x, p_i) < \epsilon \}$ with a sufficiently small $\epsilon>0$ such that $U_i' \cap U_j' =\emptyset $ for any $i \neq j$. Let $U_i \coloneqq U_i' \cap S$. By the Mayer-Vietoris exact sequence, we have the following exact sequence of homologies:
\[
\xymatrix@=10pt{
&{\phantom{aaaaaaaaaaaaaaa}\cdots }\ar[r] 
&{H_{2}(S \setminus \Sing\,S, \Z)} \oplus {\bigoplus_{i=1}^n}{H_{2}(U_i, \Z)} \ar [r]
&H_{2}(S, \Z) 
\\
\ar [r] 
&{\bigoplus_{i=1}^{n}} {H_{1}(U_i \setminus \{p_i\}, \Z)} 
\ar [r]
&{H_{1}(S \setminus \Sing\,S, \Z)} \oplus {\bigoplus_{i=1}^n}{H_{1}(U_i, \Z)} \ar [r]
& {\cdots.\phantom{aaaa}} 
}
\]
Since the singularity $(S, p_i)$ is a rational double point by \cite[Proposition 1.2]{h-w81}, $U_i$ is contractible for any $i$ by Theorem \ref{thm:2.22}. Applying Lemma \ref{lem:2.18} by setting $X= W \setminus S_1, S$ as above and $r=C \setminus (C \cap S_1)$, we have $H_1(S \setminus \Sing\,S, \Z)=H_1(r \setminus (r \cap \Sing\,S), \Z)$ and $H_2(S \setminus \Sing\,S, \Z)=0$. Hence the exact sequence is rewritten as
\[
\xymatrix@=10pt{
\cdots \ar [r]
&0 \ar [r]
&H_{2}(S, \Z) \ar [r] 
&\displaystyle{\oplus_{i=1}^{n}} {H_{1}(U_i \setminus \{p_i\}, \Z)} \ar [r]
&H_1(r \setminus (r \cap \Sing\,S), \Z) \ar [r]
& {\cdots}. 
}
\]
Since $S$ and $r \setminus (r \cap \Sing\,S)$ are affine, both $H_2(S, \Z)$ and $H_1(r \setminus (r \cap \Sing\,S), \Z)$ are free $\Z$-modules by \cite[Korollar]{ham83}. Therefore $H_{1}(U_i \setminus \{p_i\}, \Z)$ is also a free module for any $i$. 

Fix $i \in \{1, \dots, n\}$. By the Hurewicz theorem, $H_{1}(U_i \setminus \{p_i\}, \Z)$ is the abelianization of the local fundamental group $\pi_{S, p_i}$ of $(S, p_i)$. By Theorem \ref{thm:2.20}, $\pi_{S, p_i}$ is finite and so is $H_{1}(U_i \setminus \{p_i\}, \Z)$. Hence $H_{1}(U_i \setminus \{p_i\}, \Z)=0$ and $\pi_{S, p_i}$ is perfect. Hence $(S, p_i)$ is the $E_8$ singularity by Theorem \ref{thm:2.21}. By Theorem \ref{thm:dP} (1), it follows that 
\[
\begin{aligned}
9&\geq (9-K_{S_2}^2)+\eu(\P^2)=\eu(\wt{S}_2) \\
 &\geq \eu(S_2)+8n\\
 &\geq 3+8n.
\end{aligned}
\]
Hence $n=0$ and we have the assertion.
\end{proof}

The next proposition is used in our proof of Propositions \ref{prop:No23} and \ref{prop:No14}. In the forthcoming article \cite{n2}, this proposition is also used to investigate the boundary divisors $D_1$ and $D_2$ of triplets $(V, D_1 \cup D_2, U)$.

\begin{prop}\label{prop:5.3} 
Suppose that the assumptions in Proposition \ref{prop:5.2} hold. Let $\pi\colon \wt{S}_2 \rightarrow S_2$ be the minimal resolution. Then $B_0(F)=1, B_1(F)=0$ and $B_2(F)=B_2(S_2) + \sharp(C \cap S_1)+2p_a(C) -1$.
\end{prop}

\begin{proof}
Since $F$ is the support of a member of $\left|-K_{S_2}\right|$ and $-K_{S_2}$ is ample, we have $B_0(F)=1$.
Applying Lemma \ref{lem:2.18} by setting $X= \P^3 \setminus S_1, S= S_2 \setminus F$ and $r=C \setminus (C \cap S_1)$, we have the following by Proposition \ref{prop:5.2}:
\[
H_i(\wt{S}_2 \setminus (\pi^{-1}_*(F) \cup E_{\pi}), \Z)=H_i(S_2 \setminus F, \Z)=
\left\{
\begin{array}{ll}
\Z &i=0\\
\Z^{\sharp(C \cap S_1)+2p_a(C)-1}&i=1\\
0&i \geq 2.
\end{array}
\right.
\]

Now we consider the following exact sequence of cohomologies:
\[
\xymatrix@=10pt{
&{\phantom{aaaaaaaaaaaaa} \cdots} \ar [r]&H^{1}(\wt{S}_2, \Z) \ar [r]
&H^{1}(\pi^{-1}_*(F) \cup E_{\pi}, \Z)\\
\ar[r]
&H^{2}(\wt{S}_2, \pi^{-1}_*(F) \cup E_{\pi}, \Z) \ar [r]
&H^{2}(\wt{S}_2, \Z) \ar [r]
&H^{2}(\pi^{-1}_*(F) \cup E_{\pi}, \Z)\\
\ar[r]
&H^{3}(\wt{S}_2, \pi^{-1}_*(F) \cup E_{\pi}, \Z) \ar [r]
&H^{3}(\wt{S}_2, \Z) \ar[r]
&{\cdots. \phantom{aaaaaaaaa}}
}
\]

By the Lefschetz duality, we have $H^{i}(\wt{S}_2, \pi^{-1}_*(F) \cup E_{\pi}, \Z) \cong H_{4-i}(\wt{S}_2 \setminus (\pi^{-1}_*(F) \cup E_{\pi}), \Z)$. We also note that $H^{1}(\wt{S}_2, \Z)=H^{3}(\wt{S}_2, \Z)=0$ since $S_2$ is rational. Hence $B_1(\pi^{-1}_*(F) \cup E_{\pi}, \Z) = 0$ and $B_2(\pi^{-1}_*(F) \cup E_{\pi}) = B_2(\wt{S}_2) + \sharp(C \cap S_1)+2p_a(C) -1$. Moreover, we have the following exact sequence of cohomologies by Lemma \ref{lem:2.17}:
\[
\xymatrix@=10pt{
\cdots \ar [r]
&{H^{1}(S_2, \Z)} \ar [r]
&H^{1}({F}, \Z) \oplus H^{1}(\wt{S}_2, \Z) \ar[r] 
&H^{1}(\pi^{-1}_*(F) \cup E_{\pi}, \Z)
\\
{\phantom{aaa}}\ar [r]
&{H^{2}(S_2, \Z)}\ar [r]
&H^{2}({F}, \Z) \oplus H^{2}(\wt{S}_2, \Z) \ar[r] 
&{H^{2}(\pi^{-1}_*(F) \cup E_{\pi}, \Z)}
\\
{\phantom{aaa}}\ar[r]
&H^{3}(S_2, \Z).
&
&
}
\]
We note that $H^{1}({S}_2, \Z)=H^{3}({S}_2, \Z)=0$ since $S_2$ is rational.
Hence $B_1(F)=0$ and 
\[
B_2(F)=B_2(S_2)+B_2(\pi^{-1}_*(F) \cup E_{\pi}) - B_2(\wt{S}_2)=B_2(S_2) + \sharp(C \cup F)+2p_a(C) -1.
\]
\end{proof}

\subsection{The case $W_1=\P^3$}\label{sec:P3}
In this subsection, we assume that $W_1=\P^3$. By (\ref{eq:H}), the surface $S_1$ (resp.\,$\ S_2$) is a hyperplane which does not contain $C$ (resp.\,a cubic hypersurface which contains $C$ such that $C \not\subset \Sing\,S_2 $).

\begin{prop}\label{prop:4.2}
It holds that $p_a(C) \leq 1$.
\end{prop}

\begin{proof}
As $B_3(V)=2p_a(C)$, it holds that 
\begin{equation}\label{eq:P3}
4 \geq \eu(F) = \eu(S_2) +2p_a(C)+ N_1 +N_2 - N_{1 \cap 2}-2 \geq 2p_a(C)
\end{equation}
by Lemma \ref{lem:deg}, Proposition \ref{prop:2.10} and Theorem \ref{thm:dP}. Hence $2 \geq p_a(C)$.

 Suppose that $p_a(C)=2$. Since all equalities of (\ref{eq:P3}) hold, we have $\eu(S_2)=1$ and hence $S_2$ is a cone over an elliptic curve by Theorem \ref{thm:dP}. 

By \cite{h-w81}, its minimal resolution $\wt{S}_2$ is a geometrically ruled surface over the elliptic curve, which corresponds to a vector bundle of degree 3. 
We write $\tau^{-1}_*(C) \equiv a C_0 +b f$, where $C_0$ (resp.\,$f$) is the minimal section (resp.\,a ruling) of $\wt{S}_2$. Then we have 
\[
\begin{aligned}
2=2p_a(\tau^{-1}_*(C))-2&=(\tau^{-1}_*(C) \cdot \tau^{-1}_*(C)+K_{\wt{S}_2})\\
 &=(aC_0+bf \cdot (a-2)C_0+(b-3)f)\\
 &=(2b-3a)(a-1)
\end{aligned}
\]
by the genus formula. Hence $(a, b)=(2,4)$ or $(3, 5)$, which implies that $\tau^{-1}_*(C)$ is reducible. This contradicts the irreducibility of $C$. 
\end{proof}

By \cite[Table 2]{m-m81} and Proposition \ref{prop:4.2}, $C$ is either a smooth rational curve of degree $1 \leq d \leq 4$ or an elliptic curve of degree $3 \leq d \leq 5$.
The following completes the proof of Theorem~\ref{thm:main}~(1) in this case.

\begin{prop}\label{prop:No17}
The center of the blow-up $C$ cannot be an elliptic curve of degree $5$.
\end{prop}

\begin{proof}
Suppose that $C$ is an elliptic curve of degree 5. By Lemma \ref{lem:deg}, Proposition \ref{prop:2.10} and Theorem \ref{thm:dP}, we have 
\begin{equation}\label{eq.No.17}
4 \geq \eu(F) \geq \eu(S_2) + 1 \geq 2.
\end{equation}

Suppose that $S_2$ is not a cone. Then all equalities of (\ref{eq.No.17}) hold by Theorem \ref{thm:dP}. Hence $\eu(S_2)=3$ and $F$ is a sum of three concurrent lines on $S_1 \cong \P^2$.
If $S_2$ is normal, then each sum of three $(-1)$-curves is not concurrent by \cite[Appendix: Configurations of the Singularity types]{qi02}, a contradiction. 
If $S_2$ is non-normal, then $S_2$ belongs to the class (C) of \cite{a-f83} and the conductor locus $\ol{E}_{S_2}$ of the normalization of $S_2$ is reducible, which contradicts Lemma \ref{lem:line}. 

Hence $S_2$ must be a cone. By \cite{h-w81} and \cite{a-f83}, the resolution $\tau\colon \wt{S}_2 \rightarrow S_2$ contracts the minimal section $C_0$ of the geometrically ruled surface $\wt{S}_2$, which corresponds to a vector bundle of degree 3 on a curve. We also note that $\tau^*(-K_{S_2}) \sim C_0 + 3f$, where $f$ is a ruling of $\wt{S}_2$. We can write $\tau^{-1}_*(C) \sim aC_0 +bf$ with $3b \geq a >0$ since $C$ is not a ruling of $S_2$. 
Then we have
\[
5=(C \cdot -K_{S_2})_{S_2}=(\tau^{-1}_*(C) \cdot \tau^*(-K_{S_2}))_{\wt{S}_2} = (a C_0 + bf \cdot C_0 + 3f)_{\wt{S}_2}=b.
\] 
This implies $\tau^{-1}_*(C) \sim C_0+5f$. Hence $C$ must be singular since $(\tau^{-1}_*(C) \cdot C_0)=2$, which contradicts the smoothness of $C$.
\end{proof}

\subsection{The case $W_1=\Q^3$}\label{sec:Q3}
In this subsection, we assume that $W_1=\Q^3$. By (\ref{eq:H}), the surface $S_1$ (resp.\,$S_2$) is a hyperplane section which does not contain $C$ (resp.\,a member of $\left| \mathcal{O}_{\Q^3}(2) \right|$ which contains $C$ such that $C \not\subset \Sing\,S_2$).

\begin{prop}\label{prop:9.2}
It holds that $p_a(C)\leq 1$.
\end{prop}

\begin{proof}
Note that $\eu(S_1) \geq 3$ since $S_1$ is either $\P^1 \times \P^1$ or $\Q^2_0$. Since $
B_3(\Q^3)=0$, we obtain the desired inequality 
\[
5 \geq\eu(F) \geq \eu(S_1)+\eu(S_2) +2p_a(C)-4 \geq 2p_a(C)+2
\]
by Theorem \ref{thm:dP}, Lemma \ref{lem:cone}, Proposition \ref{prop:2.10} and Lemma \ref{lem:deg}.
\end{proof}

\begin{lem}\label{lem:F_0}
If $S_1 \cong \P^1 \times \P^1$, then we have $\eu(F) \leq 4$.
\end{lem}

\begin{proof}
Suppose that $B_2(F)=4$. By the proof of Lemma \ref{lem:deg}, it suffices to show that $B_1(F)>0$.
Since $F$ is the support of a member of $\left|\mathcal{O}_{\P^1 \times \P^1}(2, 2)\right|$, $F$ consists of four rulings $f_{ij}$ for $1 \leq i, j \leq 2$ with $(f_{kj} \cdot f_{lj})=\delta_{kl}$. Hence $B_1(F)=1$.
\end{proof}

By \cite[Table 2]{m-m81} and Proposition \ref{prop:9.2},
$C$ is either a smooth rational curve of degree $1 \leq d \leq 4$ or an elliptic curve of degree $4 \leq d \leq 5$.
To prove Theorem \ref{thm:main} (1), we have only to show the following:

\begin{prop}
The center of the blow-up $C$ cannot be an elliptic curve of degree 5.
\end{prop}

\begin{proof}
Suppose that $C$ is an elliptic curve of degree 5. Then $V$ is of No.17 in \cite[Table 2]{m-m81} and hence given by the blow-up $\varphi_2\colon V \rightarrow \P^3$ of $\P^3$ along an elliptic curve of degree 5 (see \cite[\S I\hspace{-.1em}I\hspace{-.1em}I-3]{ma95}). Applying Proposition \ref{prop:No17} by exchanging the subscripts of $\varphi_1$ and $\varphi_2$, we have the assertion.
\end{proof}

\begin{prop}\label{prop:No23}
The center of the blow-up $C$ cannot be an elliptic curve of degree 4.
\end{prop}

\begin{proof}
Suppose that $C$ is an elliptic curve of degree 4. Since $C$ is a complete intersection of divisors in $\left| \mathcal{O}_{\Q^3}(1)\right|$ and $\left|\mathcal{O}_{\Q^3}(2)\right|$, the surface $S_2$ is normal. Also, $S_2$ is rational by \cite{h-w81} and Lemma \ref{lem:cone}. By the same method as in Proposition \ref{prop:9.2}, we obtain
\[
5 \geq \eu(F) \geq \eu(S_1) + \eu(S_2) +2p_a(C)-4 \geq \eu(S_1)+1.
\]
Therefore $S_1$ cannot be $\P^1 \times \P^1$ by Lemma \ref{lem:F_0}. Hence $S_1 \cong \Q^2_0$ and $\Q^3 \setminus S_1 \cong \A^3$ by \cite[Theorem A]{fur93}. We can rewrite the above inequality as
\[
5 \geq \eu(F) \geq \eu(S_2) +2p_a(C)-1 = \eu(S_2)+1 \geq 4.
\]
By Proposition \ref{prop:5.3}, one of the following two cases occurs:
\begin{enumerate}
\item[\textup{(1)}] $B_1(F)=0, B_2(F)=4, B_2(S_2) \leq 2$ or
\item[\textup{(2)}] $B_1(F)=0, B_2(F)=3, B_2(S_2) = 1$.
\end{enumerate}

For each case, the intersection $F$ is a sum of concurrent lines since $\Q^2_0$ contains $F$. However, each sum of four $(-1)$-curves in $S_2$ is not concurrent when $B_2(S_2) \leq 2$ and there is at most two $(-1)$-curves in $S_2$ when $B_2(S_2)=1$ by \cite[Appendix: Configurations of the Singularity types]{qi02}. Hence we have a contradiction to the existence of $F$.
\end{proof}

\subsection{The case $W_1=V_5$}\label{sec:V5}
In this subsection, we assume that $W_1=V_5$. By (\ref{eq:H}), the surface $S_1$ (resp.\,$S_2$) is a hyperplane section which does not contain $C$ (resp.\,a hyperplane section which contains $C$ such that $C \not\subset \Sing\,S_2$). By \cite[Table 2]{m-m81}, $C$ is either a smooth rational curve of degree $1 \leq d \leq 3$ or a complete intersection of two hyperplane sections.
To prove Theorem \ref{thm:main} (1), we have only to show the following:
\begin{prop}\label{prop:No14}
The center of the blow-up $C$ cannot be a complete intersection of two hyperplane sections.
\end{prop}

\begin{proof}
Suppose that $C$ is a complete intersection of two hyperplane sections. Then the surface $S_2$ is normal and rational by Lemma \ref{lem:cone} and \cite{h-w81}. By Proposition \ref{prop:2.10} and Lemma \ref{lem:deg}, we have
\begin{equation}\label{eq:No14}
\begin{aligned}
6 \geq 1 + B_2(F) &\geq \eu(S_1) + \eu(S_2)+N_1 +N_2 -N_{1 \cap 2}-3 \\
 &\geq \eu(S_1) + \eu(S_2)-2.
\end{aligned}
\end{equation}
\begin{claim}
It holds that $\eu(S_1) =3$.
\end{claim}
\begin{claimproof}
By Theorem \ref{thm:dP} (4) and Lemma \ref{lem:cone}, we have $\eu(S_i) \geq 3$ for $i=1,2$. Suppose that $\eu(S_1)\geq 4$. By (\ref{eq:No14}), we have $B_2(F) \geq 4$ and hence $F \subset S_2$ has at least three lines. By \cite[Appendix: Configurations of the Singularity types]{qi02}, the surface $S_2$ contains only one line if $\eu(S_2)=3$.
Hence $\eu(S_2) \geq 4$. By (\ref{eq:No14}), we have $\eu(S_2)=4$ and $F$ is a sum of 5 lines.
This contradicts the fact that $S_2$ contains at most three lines by \cite[Appendix: Configurations of the Singularity types]{qi02}.
\hfill $\blacksquare$
\end{claimproof}
Hence $S_1$ is either normal with an $A_4$-singularity which contains only one $(-1)$-curve by \cite[Appendix: Configurations of the Singularity types]{qi02} or non-normal belonging to the class (C) of \cite{a-f83}. Hence $V_5 \setminus S_1 \cong \A^3$ by \cite[Theorem A]{fur93}. By Proposition \ref{prop:5.3}, one of the following three cases occurs:
\begin{enumerate}
\item[\textup{(1)}] $B_2(F)=5$
\item[\textup{(2)}] $B_2(F)=4, B_2(S_2) \leq 2, $
\item[\textup{(3)}] $B_2(F)=3, B_1(F)=0, B_2(S_2)=1$
\end{enumerate}

If the case (1) occurs, then $F$ is a sum of 5 lines. However, $S_1$ cannot contain such $F$ by \cite[Appendix: Configurations of the Singularity types]{qi02} and Lemma \ref{lem:line}, a contradiction.

If the case (2) occurs, then $F$ contains 3 lines. Hence $S_1$ must be non-normal and $S_2$ contains at most three lines. Hence $F \subset S_1$ contains a conic, but $S_1$ contains no conic (see \cite[Lemma 4.5, 4.6]{ki05}), a contradiction.

Hence the case (3) must occur. Since $S_2$ contains the unique $(-1)$-curve, $F$ is a sum of the line and two conics. In particular, the intersection $S_1 \cap S_2$ is reduced. Since the line contains an $A_4$-singular point $p_1 \in S_2$, we obtain $p_1 \in \Sing\,F$. 

Note that $N_1+N_2-N_{1 \cap 2}=1$ by (\ref{eq:No14}) and hence $N_1=\sharp (C \cap F)=1$ by Proposition \ref{prop:2.10} (2). Let $p_2\in C \cap F$ be the point. Since $C$ is an ample divisor on $S_2$, all irreducible component of $F$ must contain $p_2$. Since $B_1(F)=0$, the point $p_2$ is the only concurrent point on $F$ and hence we have $p_1=p_2$.
Then the smooth Cartier divisor $C$ on $S_2$ contains a singular point $p_1$, which gives a contradiction.
\end{proof}

\section{Exclusion of primitive Fano 3-folds}\label{sec:excpri}
The aim of this section is to prove Theorem \ref{thm:main} (1) when $V$ is primitive.
Hence it suffices to show that the following cases cannot occur (see \cite[\S I\hspace{-.1em}I\hspace{-.1em}I-3]{ma95}):
\begin{itemize}
\item $V$ is a Fano 3-fold of No.2. 
It is given by a double covering $g \colon V \rightarrow \P^1 \times \P^2$ whose branch locus $B$ is a divisor of bidegree $(2, 4)$. 
One extremal contraction $\varphi_1=\mathrm{pr}_1 \circ g$ of $V$ is of $D_1$-type and the other $\varphi_2=\mathrm{pr}_2 \circ g$ is of $C_1$-type. 
We also have $B_3(V)=40$.
\item $V$ is a Fano 3-fold of No.6. 
Its extremal contractions $\varphi_1$ and $\varphi_2$ are $C_1$-type of discriminant degree 6. 
We also have $B_3(V)=18$.
\item $V$ is a Fano 3-fold of No.8. 
It is given by a double covering $g \colon V \rightarrow V_7\coloneqq \P_{\P^2}(\mathcal{O} \oplus \mathcal{O}(1))$ whose branch locus is a member of $\left| -K_{V_7} \right|$. 
Note that $V_7$ has extremal contractions $\psi_1\colon V_7 \rightarrow \P^3$ and $\psi_2\colon V_7 \rightarrow \P^2$. 
One extremal contraction $\varphi_1$ of $V$ is the Stein factorization of $\psi_1 \circ g$, which is of $E_3$ or $E_4$-type, and the other $\varphi_2$ is the composite $\psi_2 \circ g$, which is of $C_1$-type. 
We also have $B_3(V)=18$.
\item $V$ is a Fano 3-fold of No.18. 
It is given by a double covering $g \colon V \rightarrow \P^1 \times \P^2$ whose branch locus $B$ is a divisor of bidegree $(2, 2)$.
One extremal contraction $\varphi_1=\mathrm{pr}_1 \circ g$ of $V$ is of $D_2$-type, and the other $\varphi_2=\mathrm{pr}_2 \circ g$ is of $C_1$-type.
We also have $B_3(V)=4$.
\item $V$ is a Fano 3-fold of No.32. It is a divisor on $\P^2 \times \P^2$ of bidegree $(1,1)$.
\item $V$ is a Fano 3-fold of No.35. It is the blow-up of $\P^3$ at a point.
\end{itemize}
In what follows, we exclude the possibilities of above six cases separately.
\begin{prop}
The Fano 3-fold $V$ cannot be of either No.32 or No.35.
\end{prop}

\begin{proof}
Each of Fano 3-folds of No.32 and No.35 has different extremal contractions of length two. Hence we have the assertion by Proposition \ref{prop:2.6} (2).
\end{proof}

\begin{prop}
The Fano 3-fold $V$ cannot be of No.2.
\end{prop}

\begin{proof}
Suppose that $V$ is a Fano 3-fold of No.2. We may assume that $D_1 \sim H_1$ and $D_2 \sim H_2$ by Proposition \ref{prop:2.6}. 

Let us compute topological Euler characteristics of boundaries. As $g(D_1) \cong \P^2$ and $B|_{g(D_1)} \sim \mathcal{O}_{\P^2}(4)$, it holds that 
\[
\eu(D_1) =2\eu({g(D_1)})-\eu(B|_{g(D_1)}) \geq 6-5=1.
\]
As $g(D_2) \cong \P^1 \times \P^1$ and $B|_{g(D_2)} \sim \mathcal{O}_{\P^1 \times \P^1}(2,4)$, it holds that
\[
\eu(D_2) = 2\eu({g(D_2)})-\eu(B|_{g(D_2)}) \geq 8-(1+2+4)=1.
\]
As $-K_V$ is ample, it holds that 
\[
\eu(D_1 \cap D_2) \leq B_0(D_1 \cap D_2)+B_2(D_1 \cap D_2) \leq 2(D_1 \cdot D_2 \cdot -K_V) =4.
\]
Hence we obtain the following inequality which contradicts Lemma \ref{lem:euler}.
\[
\eu(D_1 \cap D_2) \leq 4 < 37 \leq \eu(D_1) + \eu(D_2) +B_3(V)-5.
\]
\end{proof}

\begin{prop}
The Fano 3-fold $V$ cannot be of No.6.
\end{prop}

\begin{proof}
Suppose that $V$ is a Fano 3-fold of No.6. We may assume that $D_1 \sim H_1$ and $D_2 \sim H_2$ by Proposition \ref{prop:2.6}. 

Let us compute topological Euler characteristics of boundaries. Fix $i\in\{1,2\}$. Then $\varphi_i|_{D_i}$ is a conic bundle. 
If the image of $\varphi_i|_{D_i}$ is not contained in the discriminant locus, then the general fiber of $\varphi_i|_{D_i}$ is smooth and we have 
\[
\eu(D_i) \geq \eu(\P^1) \times \eu(\P^1)=4.
\] 
Otherwise, the general fiber of $\varphi_i|_{D_i}$ is a reducible conic and the special fiber is a non-reduced line. Since $\varphi_i|_{D_i}$ has at most $6-1=5$ special fiber, we have 
\[
\eu(D_i) \geq 3 \times (\eu(\P^1)-5) +2 \times 5=1.
\]
Since $-K_V$ is ample, we have 
\[
\eu(D_1 \cap D_2) \leq B_0(D_1 \cap D_2)+B_2(D_1 \cap D_2) \leq 2(D_1 \cdot D_2 \cdot -K_V) =8.
\] 
Hence we obtain the following inequality which contradicts Lemma \ref{lem:euler}. 
\[
\eu(D_1 \cap D_2)\leq 8 < 15 \leq \eu(D_1) + \eu(D_2) +B_3(V)-5.
\]
\end{proof}
\begin{prop}
The Fano 3-fold $V$ cannot be of No.8.
\end{prop}

\begin{proof}
Suppose that $V$ is a Fano 3-fold of No.8. We may assume that $D_1 \sim H_1$ and $D_2 \sim H_2$ by Proposition \ref{prop:2.6}.

Let us compute topological Euler characteristics of the boundary divisors.
As $g(D_1) = \P^2 \sim \psi_1^*\mathcal{O}_{\P^3}(1)$ and $B|_{g(D_1)} \sim \mathcal{O}_{\P^2}(4)$, it holds that
\[
\eu(D_1) = 2 \eu(\P^2) -\eu(B|_{g(D_1)}) \geq 6-5=1.
\]
As $g(D_2) = \F_1$ and $B|_{g(D_2)} \sim 2\Sigma_1+4f_1$, it holds that 
\[
\eu(D_2) = 2 \eu(\F_1) -\eu(B|_{g(D_2)}) \geq 8-6=2.
\]
As $g(D_1) \cap g(D_2) \cong \P^1$, it holds that 
\[
\eu(D_1 \cap D_2)= 2 \eu(\P^1) -\eu(g(D_1) \cap g(D_2) \cap B) \leq 4.
\]
Hence we obtain the following inequality which contradicts Lemma \ref{lem:euler}.
\[
\eu(D_1 \cap D_2)\leq 4 < 16 \leq \eu(D_1) + \eu(D_2) +B_3(V)-5.
\]
\end{proof}

\begin{prop}\label{prop:No18}
The Fano 3-fold $V$ cannot be of No.18.
\end{prop}

\begin{proof}
Suppose that $V$ is a Fano 3-fold of No.18. Since the length of $\varphi_2$ is one, we may assume that $D_1 \sim H_1+H_2$ and $D_2 \sim H_2$ by Proposition \ref{prop:2.6}. 

Let us compute topological Euler characteristics of $D_2$ and $D_1 \cap D_2$.
Since $-K_{D_2} \sim (H_1 +H_2)|_{D_2}$, the surface $D_2$ is a Gorenstein del Pezzo surface of degree $(-K_{D_2})^2 =4$. As both $H_1 +H_2$ and $D_1|_{D_2} \sim -K_{D_2}$ are ample, we have
\begin{equation}\label{eq:No18.0}
\eu(D_1 \cap D_2) \leq 1+B_2(D_1 \cap D_2) \leq 1+(D_1 \cdot D_2 \cdot H_1 +H_2) =5.
\end{equation}

\begin{claim}\label{cl:6.4}
It holds that $\eu(D_2) \geq 4$. Moreover, $\eu(D_2) =4$ implies that $\eu(D_1 \cap D_2) \leq 4$.
\end{claim}
\begin{claimproof}
Suppose that $D_2$ is normal. Then the general fiber of the conic bundle $\varphi_2|_{D_2}$ is smooth. Hence $\eu(D_2) \geq \eu(\P^1) \times \eu(\P^1)=4$.

Suppose that $D_2$ is non-normal. Let $\sigma_{D_2} \colon \ol{D}_2 \rightarrow D_2$ be the normalization. Then $D_2$ is not a cone by Lemma \ref{lem:cone} and hence $D_2$ belongs to 
\[
\left\{
\begin{array}{ll}
\textup{the class (B) of \cite{a-f83}}&\textup{if } \ol{D}_2 \cong \P^2, \\
\textup{the class (C) of \cite{a-f83}}&\textup{if } \ol{D}_2 \cong \F_2 \textup{ or}\\
\textup{the class (D) of \cite{a-f83}}&\textup{if } \ol{D}_2 \cong \F_0. 
\end{array}
\right.
\]
We note that $\sigma_{D_2}^*(H_i|_{D_2})$ is a nef and effective divisor in $\ol{D}_2$ with self-intersection number 0 for $i=1,2$. As $(H_1 \cdot H_2 \cdot D_2)=1$, we have $\sigma_{D_2}^*(H_1|_{D_2}) \not \sim \sigma_{D_2}^*(H_2|_{D_2})$. Hence only the case that $\ol{D}_2 \cong \F_0$ can occur and $D_2$ belongs to the class (D). By \cite{a-f83}, $\sigma_{D_2}$ is an isomorphism without one ruling $\Sigma_0$ and $\sigma_{D_2}|_{\Sigma_0}$ is a double covering to $\P^1$. Hence $\eu(D_2)=4$ and we have the first assertion.

We prove the second assertion below. To seek a contradiction, we assume that $\eu(D_2)=4$ and $\eu(D_1 \cap D_2)=5$. Then $B_1(D_1 \cap D_2)=0$ and $D_1 \cap D_2$ consists of four curves whose intersection number with $(H_1+H_2)|_{D_2} \sim -K_{D_2}$ is one by (\ref{eq:No18.0}).

Suppose that $D_2$ is normal. 
Then $D_1 |_{D_2} \subset D_2$ is a sum of four $(-1)$-curves with trivial first Betti number.
However, \cite[Appendix: Configurations of the Singularity types]{qi02} shows us that every sum of four $(-1)$-curves in $D_2$ has non-trivial first Betti number, a contradiction.

Suppose that $D_2$ is non-normal. Then we have shown that $D_2$ belongs to the class (D). 
By \cite{a-f83}, we have $\sigma_{D_2}^*(-K_{D_2}) \sim \Sigma_0+2f_0$. Then ${\sigma_{D_2*}}^{-1}(D_1|_{D_2})$ consists of four curves whose intersection number with $\Sigma_0+2f_0$ is one. Hence each of four curves is linearly equivalent to $f_0$. However, this implies that $D_1|_{D_2}$ is disconnected since $\sigma_{D_2}$ is as stated above, which contradicts the ampleness of $D_1|_{D_2}$.
\hfill $\blacksquare$
\end{claimproof}

By Lemma \ref{lem:euler} and Claim \ref{cl:6.4}, we have 
\[
\eu(D_1)=\eu(D_1 \cap D_2)-\eu(D_2)-B_3(V)+5 \leq 1,
\]
but we obtain $\eu(D_1) \geq 2$ by Lemmas \ref{lem:No18.1}--\ref{lem:No18.5} below, a contradiction.
\end{proof}

Here we explain how to deduce $\eu(D_1)\geq 2$ from Lemmas \ref{lem:No18.1}--\ref{lem:No18.5}.
If it holds that $\eu(D_1)\leq 3$, then $D_1$ is a certain non-normal surface (Lemma \ref{lem:No18.1}). Suppose $D_1$ is such a certain non-normal surface. Then its normalization is a $\P^1$-bundle over a smooth curve (Lemma \ref{lem:No18.2}) and the arithmetic genus $g_a$ of the curve is either $0$ or $1$ (Lemmas \ref{lem:nnwdP}--\ref{lem:No18.3}). We can compute $\eu(D_1)\geq 3$ when $g_a=0$ (Lemma \ref{lem:No18.4}) and $\eu(D_1)=2$ when $g_a=1$ (Lemma \ref{lem:No18.5}). Hence the inequality $\eu(D_1)\geq 2$ always holds. 

\begin{lem}\label{lem:No18.1}
If $\eu(D_1) \leq 3$, then $D_1$ is non-normal with $(E_{D_1} \cdot H_1 )>0$, where $E_{D_1}$ is the conductor locus of $D_1$.
\end{lem}

\begin{proof}
As $D_1 \sim H_1+H_2$ and $\varphi_1$ is a conic bundle, so is $\varphi_1|_{D_1}$. Suppose that $D_1$ is normal. Then a general fiber of $\varphi_1|_{D_1}$ is smooth and hence $\eu(D_1) \geq \eu(\P^1) \times \eu(\P^1)=4$. The same conclusion holds when $D_1$ is non-normal with $(E_{D_1} \cdot H_1 )=0$.
\end{proof}

\begin{lem}\label{lem:No18.2}
Suppose that $D_1$ is non-normal with $(E_{D_1} \cdot H_1)>0$. Let $\sigma\colon \ol{D}_1 \rightarrow D_1$ be the normalization and $\ol{E}_{D_1} \subset \ol{D}_1$ the conductor locus of $\ol{D}_1$. Then $\ol{D}_1$ is a $\P^1$-bundle over a smooth curve $Z$. 
\end{lem}

\begin{proof}
Let $\pi\colon \wt{D}_1 \rightarrow \ol{D}_1$ be the minimal resolution and $\tau \coloneqq \sigma \circ \pi$. Let $A$ be an effective divisor on $\wt{D}_1$ such that $K_{\wt{D}_1} \sim \pi^*(K_{\ol{D}_1}) - A$. Then we have 
\[
\tau^*\left((H_1+H_2)|_{D_1}\right)+K_{\wt{D}_1} 
\sim \tau^*(H_1|_{D_1}-K_{D_1})+K_{\wt{D}_1}
\sim \tau^*(H_1|_{D_1})-A - \pi^* \ol{E}_{D_1}.
\] 
As $(E_{D_1} \cdot H_1)>0$, it holds that 
\[
\begin{aligned}
&\left(\tau^*((H_1+H_2)|_{D_1})+K_{\wt{D}_1} \cdot \tau^*(H_1|_{D_1})\right)_{\wt{D}_1} \\
=& \left(\tau^*(H_1|_{D_1})-A - \pi^* \ol{E}_{D_1} \cdot \tau^*(H_1|_{D_1})\right)_{\wt{D}_1}\\ 
\leq& -\left(\pi^* \ol{E}_{D_1} \cdot \tau^*H_1|_{D_1}\right)_{\wt{D}_1}=-\left(\ol{E}_{D_1} \cdot \sigma^*H_1|_{D_1}\right)_{\ol{D}_1} <0.
\end{aligned}
\]
Hence $\tau^*\left((H_1+H_2)|_{D_1}\right)+K_{\wt{D}_1}$ is not nef and there is an extremal ray $R \subset \ol{\mathrm{NE}}(\wt{D}_1/\P^1)$ with respect to $\varphi_1 \circ \tau$ such that
\[
\left(\tau^*\left(\left(H_1+H_2\right)|_{D_1}\right)+K_{\wt{D}_1} \cdot R\right)_{\wt{D}_1} <0.
\]
Let $g\colon \wt{D}_1 \rightarrow Z$ be an extremal contraction which corresponds to $R$. Since $\varphi_1 \circ \tau\colon \wt{D}_2 \rightarrow \P^1$ is dominant, we have $\mathrm{dim}\:Z=1$ or $2$.

Suppose that $\mathrm{dim}\:Z=2$. Then there is a curve $C \subset \wt{D}_1$ such that $C^2<0$ and $(C \cdot K_{\wt{D}_1})_{\wt{D}_1} <-\left(C \cdot \tau^*\left((H_1+H_2)|_{D_1}\right)\right)_{\wt{D}_1} \leq 0$. Hence $C$ is a $(-1)$-curve such that $\left(C \cdot \tau^*\left((H_1+H_2)|_{D_1}\right)\right)_{\wt{D}_1}=(\pi_*(C) \cdot \sigma^*((H_1+H_2)|_{D_1}))_{\ol{D}_1}=0$. Since $\sigma^*((H_1+H_2)|_{D_1})$ is ample, the morphism $\sigma$ contracts $C$, which contradicts the minimality of $\sigma$. 

Hence $\mathrm{dim}\:Z=1$ and $g\colon \wt{D}_1 \rightarrow Z$ is a $\P^1$-bundle. 
Let $f$ be a ruling of $\wt{D}_1$. Then we have 
\[
-2 = (f \cdot K_{\wt{D}_1})_{\wt{D}_1} <-\left(f \cdot \tau^*\left(\left(H_1+H_2\right)|_{D_1}\right)\right)_{\wt{D}_1} \leq 0.
\]
As $f^2=0$, the morphism $\pi$ does not contract $f$. We note that $\tau^*(H_2|_{D_1})$ is a sum of rulings. Hence
\begin{equation}\label{eq:7.0.1}
\begin{aligned}
&(f \cdot -\tau^*K_{D_1})_{\wt{D}_1} =(f \cdot \tau^*((H_1+H_2)|_{D_1}))_{\wt{D}_1} = 1 \textup{ and }\\
&(f \cdot A + \pi^* \ol{E}_{D_1})_{\wt{D}_1}=-(f \cdot \tau^*\left(\left(H_1+H_2\right)|_{D_1}\right)+K_{\wt{D}_1})_{\wt{D}_1}=1.
\end{aligned}
\end{equation}
Since $(\tau^*(H_1|_{D_1}) \cdot \pi^*\ol{E}_{D_1})_{\wt{D}_1} >0$, it holds that 
\begin{equation}\label{eq:claim4}
(f \cdot \pi^* \ol{E}_{D_1})_{\wt{D}_1}=1 \textup{ and } (f \cdot A)_{\wt{D}_1}=0.
\end{equation}
Then $A$ is a sum of rulings. Hence $A=0$ and $\pi$ is crepant.
As $g\colon \wt{D}_1 \rightarrow Z$ is $\P^1$-bundle, the surface $\ol{D}_1$ is smooth unless $\ol{D}_1 \cong \Q^2_0$ and $\wt{D}_1 \cong \F_2$. However, $\ol{D}_1 \not \cong \Q^2_0$ because $\varphi_1 \circ \sigma \colon \ol{D}_1 \rightarrow \P^1$ is dominant. Hence $\pi$ is the identity and we have the assertion by regarding $g$ as a morphism from $\ol{D}_1$. 
\end{proof}

\begin{lem}\label{lem:nnwdP}
Let $S$ be a non-normal weak del Pezzo surface and $\sigma \colon \ol{S} \rightarrow S$ the normalization of $S$. Let $E$ (resp.\,$\ol{E}$) be the conductor locus of $S$ (resp.\,$\ol{S}$). Then $\mathrm{dim}H^0( \mathcal{O}_{\ol{E}})=1$. Moreover, one of the following two cases occurs:
\begin{enumerate}
\item[\textup{(1)}] $\mathrm{dim}H^1( \mathcal{O}_{\ol{E}})=0, (\sigma^*K_S \cdot \ol{E})_{\ol{S}}=-2, (K_S\cdot E)_{S}=-1$
\item[\textup{(2)}] $ \mathrm{dim}H^1( \mathcal{O}_{\ol{E}})=1, (\sigma^*K_S\cdot \ol{E})_{\ol{S}}=0, (K_S\cdot E)_{S}=0$
\end{enumerate}
\end{lem}

\begin{proof}
This follows from the same argument as in the proof of \cite[Lemma 3.35]{mo82}.
\end{proof}

\begin{lem}\label{lem:No18.3}
Under the assumptions of Lemma \ref{lem:No18.2}, it holds that $(f \cdot \ol{E}_{D_1})_{\ol{D}_1}$ $= -(f \cdot \sigma^*K_{D_1})_{\ol{D}_1}=1$, $\sigma^*(H_1|_{D_1}) \equiv 2f$ and $g_a(Z)=g_a(\ol{E}_{D_2}) \leq 1$, 
\end{lem}

\begin{proof}
The first assertion holds by (\ref{eq:7.0.1}) and (\ref{eq:claim4}). As $-(\sigma^*(H_1|_{D_1}) \cdot \sigma^*K_{D_1})_{\ol{D}_1}=(H_1 \cdot H_2 \cdot D_1)=2$, we have the second assertion. 

Since $-K_{D_1} \sim H_2|_{D_1}$ is nef and big, the surface $D_1$ is a Gorenstein weak del Pezzo surface. 
Hence the last assertion holds by Lemma \ref{lem:nnwdP} and the genus formula.
\end{proof}

\begin{lem}\label{lem:No18.4}
If $Z$ is rational, then $\eu(D_1) \geq 3$.
\end{lem}
\begin{proof}
By Lemmas \ref{lem:No18.2} and \ref{lem:No18.3}, the surface $\ol{D}_1$ is a Hirzebruch surface which contains a section $-\sigma^*K_{D_1}$ whose self-intersection number is $2$. Since $p_a(\ol{E}_{D_1})=0$, we have $-(\sigma^*K_{D_1} \cdot E_{D_1})_{\ol{D}_1}=2$ by Lemma \ref{lem:nnwdP}.
Hence one of the following two cases occurs:
\begin{enumerate}
\item[\textup{(1)}] $\ol{D}_1 \cong \F_0$ \textup{ and } $\sigma^*K_{D_1} \sim E_{D_1} \sim \Sigma_0+f_0$.
\item[\textup{(2)}] $\ol{D}_1 \cong \F_2$ \textup{ and } $\sigma^*K_{D_1} \sim E_{D_1} \sim \Sigma_2+2f_2$.
\end{enumerate}

Suppose that the case (1) occurs. Then $D_1$ is a non-normal del Pezzo surface of class (C) in \cite{a-f83}, which implies $E_{D_1} \cong \P^1$. Hence $\eu(D_1)=\eu(\ol{D}_1)-\eu(\ol{E}_{D_1})+\eu(E_{D_1}) \geq 3$.

Suppose that the case (2) occurs. Then we have 
\begin{equation}\label{eq:lem7.9.0}
(\ol{E}_{D_1} \cdot \sigma^*(H_1|_{D_1}))_{\ol{D}_1}=(\ol{E}_{D_1} \cdot \sigma^*(H_2|_{D_1}))_{\ol{D}_1}=2.
\end{equation}

Suppose that $\ol{E}_{D_1}$ is irreducible. Then so is $E_{D_1}$. 
Since $\sigma|_{\ol{E}_{D_1}}$ is not birational, we have $(E_{D_1} \cdot H_1) = 1$ by (\ref{eq:lem7.9.0}). 
Since $\varphi_1|_{E_{D_1}}\colon E_{D_1} \rightarrow \P^1$ is birational, we obtain $E_{D_1} \cong \P^1$ and $\eu({D}_1)=\eu(\ol{D}_1)-\eu(\ol{E}_{D_1})+\eu(E_{D_1}) = 4.$

Suppose that $\ol{E}_{D_1}$ is reducible. Then $\ol{E}_{D_1}$ consists of rulings $C_1$, $C_2$, and $\Sigma_2$. 
Since $(\Sigma_2 \cdot \sigma^*(H_2|_{D_1}))_{\ol{D}_1}=0$, the curve $\sigma(\Sigma_2)$ is an irreducible component of a fiber of the conic bundle $\varphi_2$ and hence $\sigma(\Sigma_2) \cong \P^1$.

Fix $i \in{1,2}$. Since $(C_i \cdot \sigma^*(H_1|_{D_1}))_{\ol{D}_1}=0$, we have $\sigma(\Sigma_2) \neq \sigma(C_i)$ and $\sigma(C_i)$ is an irreducible component of an intersection $D_1|_{H_1}$. 
Since $H_1$ is isomorphic to a quadric surface in $\P^3$ and $D_1|_{H_1} \sim -\frac{1}{2}K_{H_1}$, we have $\sigma(C_i) \cong \P^1$. As $(C_i \cdot \sigma^*(H_2|_{D_1}))_{\ol{D}_1}=1$, the morphism $\sigma|_{C_i}$ is an isomorphism. Hence $\sigma(C_1)=\sigma(C_2)$ since both of them are contained in $E_{D_1}$. 

Since $(\sigma(C_1) \cdot H_2)=1$, $(\sigma(\Sigma_2) \cdot H_2)=0$ and $E_{D_1}$ is connected, we have $\sharp\left(\sigma(\Sigma_2) \cap \sigma(C_1)\right)=1$. Hence $\eu({D}_1) =\eu(\ol{D}_1)-\eu(\ol{E}_{D_1})+\eu(E_{D_1}) \geq 4-4+3=3$. 
\end{proof}

\begin{lem}\label{lem:No18.5}
If $Z$ is elliptic, then $\eu(D_1) = 2$.
\end{lem}

\begin{proof}
The surface $\ol{D}_1$ is isomorophic to $\mathbb{P}_Z(\mathcal{E})$ for some normalized vector bundle $\mathcal{E}$ of rank two. 
Let $C_0$ be a minimal section of $\ol{D}_1$ and $e\coloneqq $deg$(\det \mathcal{E})$. 
By Lemmas \ref{lem:nnwdP} and \ref{lem:No18.3}, we can write $\ol{E}_{D_1} \equiv C_0+af$ and $-\sigma^*(K_{D_1}) \equiv C_0 +(e-a)f$ for $a \in \Z$. 
Since $-e+2(e-a)=(-\sigma^*(K_{D_1}))_{\ol{D}_1}^2=(H_2^2 \cdot (H_1+H_2))=2$, we obtain $\ol{E}_{D_1} \equiv C_0+(\frac{e}{2}-1)f$ and $-\sigma^*(K_{D_1}) \equiv C_0 +(\frac{e}{2}+1)f$. 

By \cite[Theorem V.2.15]{har77}, we have $e \geq -1$. Since $0 \leq (C_0 \cdot \sigma^*(H_2|_{D_1}))_{\ol{D}_1}$ $=1-\frac{e}{2}$, we have $e=0, 2$. If $e=0$, then $C_0$ is nef and $(\ol{E}_{D_1} \cdot C_0)_{\ol{D}_1} =-1$, a contradiction. Hence $e=2$ and $\ol{E}_{D_1}=C_0$.

As $(\ol{E}_{D_1} \cdot \sigma^*(H_2|_{D_1}))_{\ol{D}_1}=0$, the curve $E_{D_1}$ is an irreducible component of a conic. Hence $E_{D_1} \cong \P^1$ and $\eu(D_1) =\eu(\ol{D}_1)-\eu(\ol{E}_{D_1})+\eu(E_{D_1})=2$. 
\end{proof}

Summarizing the arguments in \S\ref{sec:excimpri} and \S\ref{sec:excpri}, we complete the proof of Theorem \ref{thm:main} (1).

\end{document}